\documentclass[10pt]{article}

\usepackage[english]{babel}
\usepackage[ansinew]{inputenc}
\usepackage{amsmath}
\numberwithin{equation}{section}
\usepackage{amsfonts}
\usepackage{amssymb}
\usepackage{amsthm}
\usepackage{enumitem}
\usepackage[a4paper]{geometry}
\usepackage[pdftex]{graphicx}

\usepackage{afterpage}

\usepackage{fancyhdr}
\usepackage{ifthen}

\newcommand\restr[2]{{
  \left.\kern-\nulldelimiterspace 
  #1 
  \right|_{#2} 
}}

\usepackage{marvosym}

\newcommand\blfootnote[1]{%
  \begingroup
  \renewcommand\thefootnote{}\footnote{#1}%
  \addtocounter{footnote}{-1}%
  \endgroup
}

\DeclareMathOperator*{\GS}{GS}
\DeclareMathOperator*{\GA}{GA}

\DeclareMathOperator*{\FGS}{FGS}
\DeclareMathOperator*{\supp}{supp}
\DeclareMathOperator*{\HGS}{HGS}
\renewcommand\Im{\operatorname{Im}}
\renewcommand\Re{\operatorname{Re}}

\def\ds{\displaystyle}
\def\R{\mathbb R}
\def\C{\mathbb C}
\def\N{\mathbb N}

\newcommand{\D}{\mathcal{D}}

\newcommand{\Sch}{\mathcal{S}}

\AtBeginDocument{
\def\MR#1{}}

\let\OLDthebibliography=\thebibliography
\def\thebibliography#1{
\OLDthebibliography{#1}
\addcontentsline{toc}{section}{\refname}}

\newtheorem{theo}{Theorem}
\numberwithin{theo}{section}
\newtheorem{prop}[theo]{Proposition}
\newtheorem{defin}[theo]{Definition}
\newtheorem{lema}[theo]{Lemma}
\newtheorem{cor}[theo]{Corollary}

\theoremstyle{definition}
\newtheorem{exam}[theo]{Example}

\begin{document}

\title{Quantizations and global hypoellipticity for pseudodifferential operators of infinite order in classes of ultradifferentiable functions}

\author{Vicente Asensio}

\date{\vspace{-0.5cm}}

\maketitle

\begin{abstract}
We study the change of quantization for a class of global pseudodifferential operators of infinite order in the setting  of ultradifferentiable functions of Beurling type. The composition of different quantizations as well as the transpose of a quantization are also analysed, with applications to the Weyl calculus. We also compare global $\omega$-hypoellipticity and global $\omega$-regularity of these classes of pseudodifferential operators.
\end{abstract}

\section{Introduction}\label{SectionIntroduction}
In the present paper, we deal with the change of quantization in the class of \emph{global} pseudodifferential operators introduced by Jornet and the author in~\cite{AJ}. The symbols are of \emph{infinite} order with exponential growth in all the variables, in contrast to the approach of Zanghirati~\cite{Z1986pseudodifferential} and Fern\'andez, Galbis, and Jornet~\cite{FGJ2005pseudo}, who treat pseudodifferential operators of infinite order in the~\emph{local} sense and infinite order only in the last variable, for Gevrey classes  and for classes of ultradifferentiable functions of the Beurling type in the sense of Braun, Meise, and Taylor~\cite{BMT}. In~\cite{AJ,FGJ2005pseudo}, the composition of two operators is given in terms of a suitable symbolic calculus. On the other hand, Prangoski~\cite{Pran} studies pseudodifferential operators of {\em global} type and infinite order for ultradifferentiable classes of Beurling and Roumieu type in the sense of Komatsu. We refer also to~\cite{C1, C2, CT, NR2010global} and the references therein to find other papers discussing pseudodifferential operators defined in global classes (especially Gelfand-Shilov classes).

The appropriate setting in the present paper and in~\cite{AJ} is the space of (non-quasianalytic) global ultradifferentiable functions defined by Bj\"orck~\cite{Bj}, characterized as those $f \in \mathcal{S}(\mathbb{R}^d)$, i.e. in the Schwartz class, such that for all $\lambda>0$ and all $\alpha \in \mathbb{N}_0^d$ both
\[ \sup_{x \in \mathbb{R}^d} e^{\lambda \omega(x)} |\partial^{\alpha} f(x)| \qquad  \mbox{and}  \qquad \sup_{\xi \in \mathbb{R}^d} e^{\lambda \omega(\xi)} |\partial^{\alpha} \widehat{f}(\xi)| \]
are finite, $\omega$ denoting a (non-quasianalytic) weight function in the sense of~\cite{BMT}. These spaces are always contained in the Schwartz class, and they equal the Schwartz class for the case $\omega(t)=\log(1+t)$, $t>0$, not considered in our setting.


The notion of~\emph{hypoellipticity} comes from the problem of determining whether a distribution solution to the partial differential equation $Pu=f$ is a classical solution or not.
The authors in~\cite{FGJ2005pseudo} provide adequate conditions for the construction of a (left)~\emph{parametrix} for their symbols, which guarantee the hypoellipticity in the desired class in~\cite{FGJ2004hypo}.
For the operators defined in~\cite{Pran}, the corresponding construction of parametrices is done in Cappiello, Pilipovi\'c, and Prangoski~\cite{CPP}. Here, we develop the method of the parametrix in Section~\ref{Parametrix} for the class of operators introduced in~\cite{AJ}, but also for every quantization of the pseudodifferential operator. In particular, we obtain a sufficient condition for any quantization of a pseudodifferential operator to be $\omega$-regular in the sense of Shubin~\cite{Shu2001pseudo} (see the definition of $\omega$-regularity at the beginning of Section~\ref{Parametrix}). In a forthcoming paper, the global parametrix method presented here will be used to define a suitable Weyl wave front set for $\Sch'_\omega(\R^d)$ and complete the characterization of global wave front sets given in~\cite{BJO2019TheGabor}.

As we mention at the beginning, one of the goals of the present paper is to extend the results in~\cite{AJ} by adapting them for a valid change of quantization for these symbols (see Sections~\ref{SectionSymbolicCalculus} and \ref{SectionTransposeAndComposition}). Namely, we follow the ideas for the change of quantization set within the framework of global symbol classes of Shubin~\cite[\S 23]{Shu2001pseudo}.
In~\cite{Pran} it is considered the change of quantization and its corresponding symbolic calculus for classes in the sense of Komatsu~\cite{KomI}, also in the Roumieu setting. Nonetheless, as pointed out in~\cite{AJ}, whenever the weight $\omega$ is under the mild condition
\[ \exists H>1: 2\omega(t) \leq \omega(Ht) + H, \qquad t>0, \]
the classes of ultradifferentiable functions are equally defined either by weights as in~\cite{BMT} or by sequences as in~\cite{KomI} (see Bonet, Meise, and Melikhov~\cite{BMM}). Thus, if the weight sequence $(M_p)_p$ satisfies only stability under ultradifferential operators, as assumed in~\cite{Pran}, our classes of symbols (and amplitudes) might not coincide with the ones defined in~\cite{Pran}. It turns out that, even only in the Beurling setting, we are discussing different cases.


Finally, in Section~\ref{SectionHypoellipticity}, inspired by Boggiatto, Buzano, and Rodino~\cite{BBR1996Global}, we show that some $\omega$-hypoelliptic symbols are stable under change of quantization and we compare the notions of $\omega$-regularity and $\omega$-hypoellipticity following the ideas of \cite{BJORegularity}.

\section{Preliminaries}\label{SectionPreliminaries}
We begin with some notation on multi-indices. Throughout the text we will denote by $\alpha = (\alpha_1,\ldots,\alpha_d) \in \mathbb{N}_{0}^{d}$ a multi-index of dimension $d$. The length of $\alpha$ is $|\alpha| = \alpha_1 + \cdots + \alpha_d$. For two multi-indices $\alpha$ and $\beta$ we write $\beta \leq \alpha$ for $\beta_j \leq \alpha_j,$ when $j=1,\ldots,d$. Moreover, $\alpha! = \alpha_1! \cdots \alpha_d!$ and if $\beta \leq \alpha$, then $\binom{\alpha}{\beta} := \frac{\alpha!}{\beta!(\alpha-\beta)!}$. For $x=(x_{1},\ldots,x_d) \in \mathbb{R}^d$, we have $x^{\alpha} = x_1^{\alpha_1} \cdots x_d^{\alpha_d}$. We write $\partial^{\alpha} = \big(\frac{\partial}{\partial x_1}\big)^{\alpha_1} \ldots \big(\frac{\partial}{\partial x_d}\big)^{\alpha_d}$ and we set
\[ D^{\alpha} = D^{\alpha_1}_{x_1} \cdots D^{\alpha_d}_{x_d}, \]
where $D^{\alpha_j}_{x_j} = (-i)^{|\alpha_j|} \big(\frac{\partial}{\partial x_j}\big)^{\alpha_j}$, $j=1,\ldots,d$.

In our setting we work with weight functions as the ones defined by Braun, Meise, and Taylor~\cite{BMT}.
\begin{defin}\label{DefinitionWeightFunction}
A \emph{non-quasianalytic weight function} $\omega:[0,+\infty[
\to[0,+\infty[$ is a continuous and increasing function which satisfies:
\begin{itemize}
\item[$(\alpha)$]
\vspace{-1.2mm} $\displaystyle\exists\ L\geq 1: \
\omega(2t)\leq L(\omega(t)+1),\ t\geq0,$
\item[$(\beta)$]
\vspace{-1.9mm} $\displaystyle \int_{1}^{+\infty}
\frac{\omega(t)}{t^2} dt < +\infty$,
\item[$(\gamma)$]
\vspace{-1.9mm} $\displaystyle \log(t) = o(\omega(t)) $ as $t \to
\infty$,
\item[$(\delta)$]
\vspace{-1.9mm} $\displaystyle \varphi_{\omega}: t \mapsto
\omega(e^{t})\ $ is convex.
\end{itemize}
We extend the weight function $\omega$ to $\mathbb{C}^d$ in a radial way: $\omega(z)=\omega(|z|)$, $z \in \mathbb{C}^d$, where $|z|$ denotes the Euclidean norm.
\end{defin}
All the results below work also if we consider the more general condition
\begin{itemize}
\item[$(\gamma')$] \vspace{-1mm} $\displaystyle \log(t) = O(\omega(t))$ as $t \to \infty$,
\end{itemize}
but for the sake of simplicity we use $(\gamma)$.

From Definition~\ref{DefinitionWeightFunction}($\alpha$) we immediately have:
\begin{equation}\label{EqOmegaDiff}
\omega(x+y) \leq L(\omega(x)+\omega(y)+1), \qquad x,y \in \mathbb{R}^d.
\end{equation}
For $z \in \C^d$ we denote $\langle z \rangle:=\sqrt{1+|z|^2}$. From~\eqref{EqOmegaDiff} we have
\begin{equation}\label{EqOmegaX}
\omega(\langle z \rangle) \leq \omega(1+|z|) \leq L\omega(z) + L(1+\omega(1)), \qquad z \in \mathbb{C}^d.
\end{equation}

\begin{defin}
Given a weight function $\omega$, the \emph{Young conjugate} $\varphi^{\ast}_{\omega}: [0,\infty[ \to [0,\infty[$ of $\varphi_{\omega}$ is defined as
\[ \varphi^{\ast}_{\omega}(t) := \sup_{s \geq 0} \{ st - \varphi_{\omega}(s) \}. \]
\end{defin}
When the weight function $\omega$ is clear or irrelevant in the context, we simply denote $\varphi_{\omega}$ and $\varphi^{\ast}_{\omega}$ by $\varphi$ and $\varphi^{\ast}$. From now on, we assume that $\omega|_{[0,1]} \equiv 0$ (in particular, this gives that $\omega(1)=0$ in formula \eqref{EqOmegaX}), which implies that $\varphi^{\ast}(0)=0$. Moreover, it is known that $\varphi^{\ast}$ is convex, the function $\varphi^{\ast}(x)/x$ is increasing for $x>0$ and $\varphi^{**}:=(\varphi^\ast)^*= \varphi$ (see \cite{BMT}). From \cite[Remark 2.8(c)]{hm} is not difficult to see (cf. \cite[Lemma A.1]{BJORealPaley}):

\begin{prop}\label{PropTechnical}
If a weight function $\omega$ satisfies $\omega(t)=o(t^{a})$ as $t\to+\infty$ for some $0<a\leq 1$, then for every $B>0$ and $\lambda >0$ there exists $C>0$ such that
\[ B^{n} n! \leq C e^{a\lambda \varphi^{\ast}(\frac{n}{\lambda})}, \qquad  n \in \mathbb{N}_{0}. \]
\end{prop}


The following result can be found in~\cite{BMT}.
\begin{lema}\label{lemmafistrella}
\begin{itemize}
\item[(1)] Let $L>0$ be such that $\omega(et) \leq L(\omega(t)+1)$. Then
\begin{equation}\label{EqLLambdaVarphi}
\lambda L^n \varphi^{\ast}\big(\frac{y}{\lambda L^n}\big) + ny \leq \lambda \varphi^{\ast}\big(\frac{y}{\lambda}\big) + \lambda \sum_{j=1}^n L^j
\end{equation}
for every $y \geq 0$, $\lambda>0$, $n\in \mathbb{N}$.
\item[(2)] For all $s,t,\lambda>0$, we have
\begin{equation}\label{EqEstimationVarphi}
2\lambda\varphi^{\ast}\big(\frac{s+t}{2\lambda}\big) \leq \lambda\varphi^{\ast}\big(\frac{s}{\lambda}\big) + \lambda\varphi^{\ast}\big(\frac{t}{\lambda}\big) \leq \lambda\varphi^{\ast}\big(\frac{s+t}{\lambda}\big).
\end{equation}
\end{itemize}
\end{lema}
We will consider without losing generality with no explicit mention that the constant $L\geq 1$ that comes from Definition~\ref{DefinitionWeightFunction}($\alpha$) fulfils the condition of Lemma~\ref{lemmafistrella}. For more results involving $\varphi^{\ast}$ see, for instance, \cite{AJ, BMT, FGJ2005pseudo} and \cite[Lemma A.1]{BJORealPaley}.

We deal with a class of global ultradifferentiable functions, which extends the classical Schwartz class with the use of weight functions. It was introduced by Bj\"orck~\cite{Bj}, but only considering a subadditive weight function $\omega$ (so the following definition is slightly more general than the given by Bj\"orck).
\begin{defin}
\label{def3}
For a weight $\omega$ as in Definition \ref{DefinitionWeightFunction} we define
$\Sch_\omega(\R^{d})$ as
the set of all $u\in L^1(\R^d)$ such that ($u$ and its Fourier transform $\widehat{u}$ belong to $C^\infty(\R^d)$ and)
\begin{itemize}
\item[(i)] for each $\lambda>0$ and $\alpha\in\N^d_0,\quad\ds
\sup_{x\in\R^d}e^{\lambda\omega(x)}|D^\alpha u(x)|<+\infty,$
\item[(ii)] for each
$\lambda>0$ and $\alpha\in\N^d_0,\quad\ds
\sup_{\xi\in\R^d}e^{\lambda\omega(\xi)}|D^\alpha\widehat{u}(\xi)|<+\infty.$
\end{itemize}
The corresponding strong dual is denoted by
$\Sch'_\omega(\R^{d})$ and is the set of all the linear and continuous
functionals $u:\Sch_\omega(\R^{d})\to\C$. We say that an element of
$\Sch'_\omega(\R^{d})$ is an {\em $\omega$-temperate ultradistribution.}
\end{defin}
The space $\Sch_\omega(\R^{d})$ has been studied for different purposes by many authors. We refer, for instance, to \cite{BJORegularity,BJORealPaley,GZ} for some examples of publications that treat different problems in the setting of the class $\Sch_\omega(\R^{d})$. We recall here \cite[Lemma 2.11]{AJ}, which will be useful below.
\begin{lema}\label{LemmaEquivalence}
If $f \in \mathcal{S}(\R^{d})$, then $f \in \Sch_\omega(\R^{d})$ if and only if for every $\lambda,\mu>0$ there is $C_{\lambda,\mu}>0$ such that for all $\alpha \in \mathbb{N}_{0}^{d}$ and $x\in\R^d,$ we have
\[ |D^{\alpha} f(x)| \leq C_{\lambda,\mu} e^{\lambda \varphi^{\ast}\big( \frac{|\alpha|}{\lambda} \big) } e^{-\mu \omega(x)}. \]
\end{lema}

From now on, $m$ denotes a real number and $0<\rho\leq 1$. In the following, we consider global symbols and global amplitudes of infinite order  defined very similarly to the ones in~\cite[Definitions 3.1 and 3.2]{AJ}. The unique difference is the factor $e^{m\omega(x,\xi)}$ in the case of symbols and $e^{m\omega(x,y,\xi)}$ in the case of amplitudes, which are more suitable for our purposes. We observe that these  definitions are \emph{equivalent} to those in~\cite{AJ}. In fact, when considering symbols for example, it is enough to use that there exist $A,B>0$ such that $A(\omega(x)+\omega(\xi)) \leq \omega(x,\xi) \leq B(\omega(x)+\omega(\xi)+1)$ for every $x,\xi \in \mathbb{R}^d$.
\begin{defin}\label{gs}
A \emph{global symbol} (of order $m$) in $\GS^{m,\omega}_{\rho}$ is a function $p(x,\xi) \in C^{\infty}(\mathbb{R}^{2d})$ such that for all $n\in \mathbb{N}$ there exists $C_n>0$ with
\[ |D^{\alpha}_x D^{\beta}_{\xi} p(x,\xi)| \leq C_n \langle(x,\xi)\rangle^{-\rho|\alpha+\beta|} e^{n\rho \varphi^{\ast}\big(\frac{|\alpha+\beta|}{n}\big)} e^{m\omega(x,\xi)}, \]
for all $\alpha,\beta \in \mathbb{N}_0^d$ and $x,\xi \in \mathbb{R}^d$.
\end{defin}

\begin{defin}\label{ga}
A \emph{global amplitude} (of order $m$) in $\GA^{m,\omega}_{\rho}$ is a function $a(x,y,\xi) \in C^{\infty}(\mathbb{R}^{3d})$ such that for all $n\in \mathbb{N}$ there exists $C_n>0$ with
\[ |D^{\alpha}_x D^{\beta}_y D^{\gamma}_{\xi} a(x,y,\xi)| \leq C_n \frac{\langle x-y \rangle^{\rho|\alpha+\beta+\gamma|}}{\langle(x,y,\xi)\rangle^{\rho|\alpha+\beta+\gamma|}} e^{n\rho \varphi^{\ast}\big(\frac{|\alpha+\beta+\gamma|}{n}\big)} e^{m\omega(x,y,\xi)}, \]
for all $\alpha,\beta,\gamma \in \mathbb{N}_0^d$ and $x,y,\xi \in \mathbb{R}^d$.
\end{defin}

%

In~\cite{AJ} we introduce global pseudodofferential operators on $\Sch_{\omega}(\R^d)$ by means of oscillatory integrals for global amplitudes as  in Definition~\ref{ga} (see~\cite[Proposition 3.3]{AJ}). It turns out that the action of a pseudodifferential operator on a function in $\Sch_{\omega}(\R^d)$ can be written as an iterated integral \cite[Theorem 3.7]{AJ} and it is continuous and linear from $\Sch_{\omega}(\R^d)$ into itself. In fact, we use these properties to state the following definition:
\begin{defin}
Given a global amplitude $a(x,y,\xi) \in \GA^{m,\omega}_{\rho}$ (as in Definition~\ref{ga}), we define the associated \emph{global pseudodifferential operator} $A:\mathcal{S}_{\omega}(\mathbb{R}^d) \to \mathcal{S}_{\omega}(\mathbb{R}^d)$ by
\[ A(f)(x) := \int \Big( \int e^{i(x-y)\cdot \xi} a(x,y,\xi) f(y) dy \Big) d\xi, \qquad f\in\Sch_\omega(\R^d). \]
\end{defin}
Moreover, this operator can be extended linearly and continuously to an operator $\widetilde{A}$ from $\Sch'_\omega(\R^d)$ into $\Sch'_\omega(\R^d)$ \cite[Proposition 3.10]{AJ}.

At some stages we need classes of ultradifferentiable functions defined in the  local sense; we refer the reader to~\cite{BMT, FGJ2005pseudo} for a theory of pseudodifferential operators of infinite order when defined in local spaces. Let $\omega$ be a weight function. For an open set $\Omega \subset \mathbb{R}^d$, we define the space of {\em ultradifferentiable functions of Beurling type} in $\Omega$ as
\[ \mathcal{E}_{\omega}(\Omega) := \big\{ f \in C^{\infty}(\Omega): |f|_{K,\lambda}<\infty \ \text{for every} \ \lambda>0, \ \text{and every} \ K \subset \Omega \ \text{compact} \big\}, \]
where
\[ |f|_{K,\lambda}:=  \sup_{\alpha \in \mathbb{N}_0^d} \sup_{x \in K}|D^{\alpha}f(x)| e^{-\lambda \varphi^{\ast}\big(\frac{|\alpha|}{\lambda}\big)}. \]
We endow such space with the Fr\'echet topology given by the sequence of seminorms $|f|_{K_n,n}$, where $(K_n)_{n}$ is any compact exhaustion of $\Omega$ and $n\in\N$. The strong dual of $\mathcal{E}_{\omega}(\Omega)$ is the space of compactly supported ultradistributions of Beurling type and is denoted by $\mathcal{E}'_{\omega}(\Omega)$.

The space of {\em ultradifferentiable functions of Beurling type with compact support} in $\Omega$ is denoted by $\D_{\omega}(\Omega)$, and it is the space of those functions $f\in \mathcal{E}_{\omega}(\Omega)$ such that its support, denoted by $\supp f$, is compact in $\Omega$. Its corresponding dual space is denoted by $\D'_{\omega}(\Omega)$ and it is called the space of {\em ultradistributions of Beurling type} in $\Omega$. The following continuous embeddings hold:
\[ \mathcal{E}'_{\omega}(\R^d)\subseteq \Sch'_{\omega}(\R^d) \subseteq \D'_{\omega}(\R^d). \]
We recall that the space $\Sch_{\omega}(\R^d)$, as well as its strong dual $\Sch'_{\omega}(\R^d)$, are stable under Fourier transform (see, for instance, Bj\"orck~\cite{Bj}).

Since the global amplitudes have exponential growth in all the variables, it becomes useful a particular kind of integration by parts to understand the behaviour of a pseudodifferential operator in this setting. Following~\cite{Pran}, but with a different point of view, we use in~\cite{AJ} entire functions with prescribed exponential growth in terms of a weight function $\omega$. The existence of this type of entire functions was proven by Braun~\cite{Br} and Langenbruch~\cite{Lan}. In several variables we have a similar result:
\begin{theo}[\cite{AJ}, Theorem 2.16]
\label{TheoLangenbruchSeveralVariables}
Let $\omega:[0,\infty[ \to [0,\infty[$ be a continuous and increasing function satisfying the conditions $(\alpha)$, $(\gamma)$, and $(\delta)$ of Definition~\ref{DefinitionWeightFunction}. Then there are a function $G \in \mathcal{H}(\mathbb{C}^{d})$ and some constants $C_1,C_2,C_3,C_4>0$ such that
\begin{enumerate}
\item[i')] $\displaystyle \log{|G(z)|} \leq \omega(z) + C_1, \quad z \in \mathbb{C}^{d}$;
\item[ii')] $\displaystyle \log{|G(z)|} \geq C_2\omega(z) - C_4, \quad z \in \widetilde{U}:= \{ z \in \mathbb{C}^{d}: |\Im(z)| \leq C_3(|\Re(z)| + 1) \}$.
\end{enumerate}
\end{theo}

We also need the notion of $\omega$-\emph{ultradifferential operator} with constant coefficients.  Let $G$ be an entire function in $\mathbb{C}^d$ with $\log|G| = O(\omega)$. For $\varphi \in \mathcal{E}_{\omega}(\mathbb{R}^{d})$, the map $T_{G}:\mathcal{E}_{\omega}(\mathbb{R}^{d}) \to \C$ given by
\[ T_{G}(\varphi) := \sum_{\alpha \in \mathbb{N}_{0}^{d}}  \frac{D^{\alpha}G(0)}{\alpha!} D^{\alpha}\varphi(0) \]
defines an ultradistribution $T_{G} \in \mathcal{E}'_{\omega}(\mathbb{R}^{d})$ with support equal to $\{0\}$. The convolution operator $G(D): \mathcal{D}'_{\omega}(\mathbb{R}^{d}) \to \mathcal{D}'_{\omega}(\mathbb{R}^{d})$ defined by $G(D)(\mu) = T_{G} \ast \mu$ is said to be an ultradifferential operator of $\omega$-class. 
\begin{prop}\label{PropPropertiesUltradiffOperator}
Let $G$ be the entire function given in~Theorem~\ref{TheoLangenbruchSeveralVariables} and $n \in \mathbb{N}$. If
\[ G^n(z) = \sum_{\alpha \in \mathbb{N}_0^n} b_{\alpha} z^{\alpha}, \qquad z \in \mathbb{C}^d \]
denotes the $n$-th power of $G$, then there exist $C,K>0$ such that
\begin{align*}
|b_{\alpha}| &\leq e^{nC} e^{-nC \varphi^{\ast}\big(\frac{|\alpha|}{nC}\big)}, \quad \alpha\in\N_0^d; \\
|G^n(\xi)| &\geq C^{-n} e^{nK\omega(\xi)}, \quad \xi \in \mathbb{R}^d.
\end{align*}
\end{prop}

The following result characterizes those operators whose kernel is a function in $\Sch_\omega(\R^d)$. These operators are fundamental to understand the symbolic calculus. The proof is standard.
\begin{prop}\label{PropExampleExtensionPseudodifferentialOperator}
Let $T:\mathcal{S}_{\omega}(\R^{d}) \to \mathcal{S}_{\omega}(\R^{d})$ be a pseudodifferential operator. The following assertions are equivalent:
\begin{enumerate}
\item[$(1)$] $T$ has a linear and continuous extension $\widetilde{T}:\mathcal{S}'_{\omega}(\R^{d}) \to \mathcal{S}_{\omega}(\R^{d})$.
\item[$(2)$] There exists $K(x,y) \in \mathcal{S}_{\omega}(\mathbb{R}^{2d})$ such that
\[ (T\varphi)(x) = \int K(x,y) \varphi(y) dy, \qquad \varphi \in \mathcal{S}_{\omega}(\R^{d}). \]
\end{enumerate}
\end{prop}
Any operator $T:\mathcal{S}_{\omega}(\mathbb{R}^d) \to \mathcal{S}_{\omega}(\mathbb{R}^d)$ which satisfies (1) or (2) of Proposition~\ref{PropExampleExtensionPseudodifferentialOperator} is called \emph{$\omega$-regularizing}.

\section{Symbolic calculus for quantizations}\label{SectionSymbolicCalculus}
We generalize the symbolic calculus developed in~\cite{AJ} for quantizations.
\begin{defin}\label{DefFormalSums}
We define $\FGS^{m,\omega}_{\rho}$ to be the set of all formal sums $\sum_{j \in \mathbb{N}_0} a_j(x,\xi)$ such that $a_j(x,\xi) \in C^{\infty}(\mathbb{R}^{2d})$ and there is $R\geq 1$ such that for every $n \in \mathbb{N}$ there exists $C_n>0$ with
\[ |D^{\alpha}_x D^{\beta}_{\xi} a_j(x,\xi)| \leq C_n \langle(x,\xi)\rangle^{-\rho(|\alpha+\beta|+j)} e^{n\rho \varphi^{\ast}\big(\frac{|\alpha+\beta|+j}{n}\big)} e^{m\omega(x,\xi)}, \]
for each $j \in \mathbb{N}_0$, $\alpha,\beta \in \mathbb{N}_0^d$, and $\log\big(\frac{\langle(x,\xi)\rangle}{R}\big) \geq \frac{n}{j}\varphi^{\ast}\big(\frac{j}{n}\big)$.
\end{defin}

\begin{defin}\label{DefEqFormalSums}
Two formal sums $\sum a_j$ and $\sum b_j$ in $\FGS_{\rho}^{m,\omega}$ are said to be \emph{equivalent}, denoted by $\sum a_j \sim \sum b_j$, if there is $R\geq 1$ such that for each $n \in \mathbb{N}$ there exist $C_n>0$ and $N_n \in \mathbb{N}$ with
\[ \Big|D^{\alpha}_x D^{\beta}_{\xi} \sum_{j < N}(a_j - b_j)\Big| \leq C_n \langle(x,\xi)\rangle^{-\rho(|\alpha+\beta|+N)} e^{n\rho \varphi^{\ast}\big(\frac{|\alpha+\beta|+N}{n}\big)} e^{m\omega(x,\xi)}, \]
for every $N \geq N_n$, $\alpha,\beta \in \mathbb{N}_{0}^d$, and $\log\big(\frac{\langle(x,\xi)\rangle}{R}\big) \geq \frac{n}{N} \varphi^{\ast}\big(\frac{N}{n} \big)$.
\end{defin}

The following construction has been carried out in~\cite{AJ} following the lines of~\cite[Theorem~3.7]{FGJ2005pseudo}. Let $\Phi$ be a compactly supported ultradifferentiable function of Beurling type in the space $\mathcal{D}_{\sigma}(\mathbb{R}^{2d})$, where $\sigma$ is a weight function which satisfies $\omega(t^{1/\rho})=O(\sigma(t))$,  as $t\to+\infty$, and
\begin{equation}\label{EqDefinitionPhi}
|\Phi(t)| \leq 1, \qquad \Phi(t)=1 \ \text{if} \ |t| \leq 2, \qquad \Phi(t)=0 \ \text{if} \ |t| \geq 3.
\end{equation}
Let $(j_n)_{n}$ be a sequence of natural numbers such that $j_n/n \to \infty$ as $n$ tends to infinity. For each $j_n \leq j < j_{n+1}$, we set
\begin{equation}\label{psi-function}
\varphi_j(x,\xi) := 1 - \Phi\Big(\frac{(x,\xi)}{A_{n,j}}\Big), \qquad A_{n,j}=Re^{\frac{n}{j}\varphi_{\omega}^{\ast}(\frac{j}{n})},
\end{equation}
where $R\geq 1$ is the constant which appears in Definition~\ref{DefFormalSums}. It is understood that $\varphi_0=1$. We have shown in \cite{AJ} that $\varphi_j \in \GS^{0,\omega}_{\rho}$. Moreover, if $\sum_j a_j \in \FGS^{m,\omega}_{\rho}$ then, by~\cite[Theorem 4.6]{AJ},
\[ a(x,\xi) := \sum_{j=0}^{\infty} \varphi_j(x,\xi) a_j(x,\xi) \]
is a global symbol in $\GS^{m,\omega}_{\rho}$, equivalent to $\sum_j a_j$ in $\FGS^{m,\omega}_{\rho}$.

Now, we extend some results in~\cite{AJ} for quantizations. In what follows, $\tau$ stands for a real number. Let $k \in \mathbb{N}_0$ denote the minimum natural number satisfying
\begin{equation}\label{EqDefinitionk}
|\tau|+|1-\tau| \leq 2^k.
\end{equation}
Furthermore, for any $m \in \mathbb{R}$ we denote
\begin{equation}\label{EqDefinitionmprime}
m' = mL^k,
\end{equation}
where $L\geq 1$ is the constant of Lemma~\ref{lemmafistrella}. We observe that $m'=m$ if and only if $0 \leq \tau \leq 1$.

\begin{lema}\label{LemmaSymbolAmplitudeTau}
If $b(x,\xi) \in \GS^{m,\omega}_{\rho}$ and $\tau \in \mathbb{R}$, then
\[ a(x,y,\xi) := b( (1-\tau)x + \tau y, \xi) \]
is a global amplitude in $\GA^{\max\{0,m'\},\omega}_{\rho}$.
\end{lema}
\begin{proof}
The following inequality is easy to check:
\[ \langle(x,y,\xi)\rangle \leq \sqrt{6}\langle \tau \rangle \langle x-y \rangle \langle((1-\tau)x+\tau y, \xi)\rangle, \qquad x,y,\xi \in \mathbb{R}^d, \  \tau \in \mathbb{R}. \]
We take $\widetilde{p} \in \mathbb{N}$ such that $\max\{ |1-\tau|, |\tau|, (\sqrt{6}\langle \tau \rangle)^{\rho} \} \leq e^{\rho \widetilde{p}}$. By assumption, for all $\lambda>0$ there exists $C_{\lambda}>0$ such that ($L\geq 1$ is the constant of Lemma~\ref{lemmafistrella})
\begin{align*}
|D^{\alpha}_x D^{\beta}_y D^{\gamma}_{\xi} a(x,y,\xi)| &\leq |1-\tau|^{|\alpha|} |\tau|^{|\beta|} C_{\lambda} \langle((1-\tau)x+\tau y, \xi)\rangle^{-\rho|\alpha+\beta+\gamma|} \times \\
& \quad \times e^{\lambda L^{2\widetilde{p}} \rho \varphi^{\ast}\big(\frac{|\alpha+\beta+\gamma|}{\lambda L^{2\widetilde{p}}}\big)} e^{m\omega((1-\tau)x+\tau y,\xi)}.
\end{align*}
The choice of $\widetilde{p}$ gives $|1-\tau|^{|\alpha|} |\tau|^{|\beta|} (\sqrt{6}\langle \tau \rangle)^{\rho|\alpha+\beta+\gamma|} \leq e^{2\widetilde{p}\rho|\alpha+\beta+\gamma|}$. Then, by~\eqref{EqLLambdaVarphi}, we get 
\[ \big[ e^{2\widetilde{p}|\alpha+\beta+\gamma|} e^{\lambda L^{2\widetilde{p}} \varphi^{\ast}\big(\frac{|\alpha+\beta+\gamma|}{\lambda L^{2\widetilde{p}}}\big)} \big]^{\rho} \leq e^{\lambda \rho \varphi^{\ast}\big(\frac{|\alpha+\beta+\gamma|}{\lambda}\big)} e^{\lambda \rho \sum_{j=1}^{2\widetilde{p}} L^j}. \]
Finally, since $\omega$ is radial and increasing,  applying $k$ times property ($\alpha$) of the weight function $\omega$, we get, for $m\geq0$,
\begin{equation}\label{EqWeightTau}
e^{m\omega((1-\tau)x+\tau y, \xi)} \leq e^{m\omega(2^k(x,y,\xi))} \leq e^{m' \omega(x,y,\xi)} e^{mL^k + mL^{k-1} + \cdots + mL}.
\end{equation}
\end{proof}

\begin{cor}\label{CorSymbolAmplitudeTau}
Let $\varphi_j$ be the function in~\eqref{psi-function}. For all $\lambda>0$ there exists $C_{\lambda}>0$ such that
\begin{equation*}
|D^{\alpha}_x D^{\beta}_y D^{\gamma}_{\xi} \varphi_j((1-\tau)x+\tau y, \xi)| \leq C_{\lambda} \langle((1-\tau)x+\tau y, \xi)\rangle^{-\rho|\alpha+\beta+\gamma|} e^{\lambda\rho \varphi^{\ast}\big(\frac{|\alpha+\beta+\gamma|}{\lambda}\big)},
\end{equation*}
for every $\alpha,\beta,\gamma \in \mathbb{N}_0^d$ and $x,y,\xi \in \mathbb{R}^d$. Hence $\varphi_{j}((1-\tau)x + \tau y, \xi) \in \GA^{0,\omega}_{\rho}$ for all $\tau \in \mathbb{R}$.
\end{cor}

Here, we generalize~\cite[Lemma 4.7]{AJ} to readapt it to our context.
\begin{lema}\label{LemmaAmplitudeRevisited}
Let $a(x,y,\xi)$ be an amplitude in $\GA^{m,\omega}_{\rho}$ and let $A$ be the corresponding pseudodifferential operator. For each $u \in \mathcal{S}_{\omega}(\mathbb{R}^d)$,
\[ A(u) = \mathcal{S}_{\omega}(\mathbb{R}^d) - \sum_{j=0}^{\infty} A_j(u), \]
where $A_j$ is the pseudodifferential operator defined by the amplitude
\[ (\varphi_j-\varphi_{j+1})((1-\tau)x+\tau y, \xi) a(x,y,\xi),\quad j \in \mathbb{N}_0. \]
\end{lema}
\begin{proof}
By Corollary~\ref{CorSymbolAmplitudeTau}, 
$(\varphi_j - \varphi_{j+1})((1-\tau)x+\tau y,\xi)a(x,y,\xi) \in \GA^{m,\omega}_{\rho}$. Since $A_{n,N+1}\to \infty$ as $N\to\infty$, proceeding as in~\cite[Proposition 3.3]{AJ}, one can show that, for each $u \in \mathcal{S}_{\omega}(\mathbb{R}^d)$, the series
\[ \sum_{j=0}^{\infty} A_j(u)(x) = \mathcal{S}_{\omega}(\mathbb{R}^d) - \lim_{N \to \infty} \iint e^{i(x-y)\cdot \xi} \big( 1-\varphi_{N+1}((1-\tau)x+\tau y,\xi)\big)a(x,y,\xi) u(y) dy d\xi \]
defines a continuous linear operator $B_\tau:\Sch_\omega(\R^d)\to\Sch_\omega(\R^d)$. It remains to show that for all $\tau \in \mathbb{R}$,
\begin{equation}\label{EqEqualityTauLimitOperator}
B_{\tau}(f)(x) = \iint e^{i(x-y)\cdot \xi} a(x,y,\xi) f(y) dy d\xi, \qquad f \in \mathcal{S}_{\omega}(\mathbb{R}^d).
\end{equation}
We recall that
\[ (1-\varphi_{N+1})((1-\tau)x+\tau y,\xi) = \Phi\Big(\frac{((1-\tau)x+\tau y,\xi)}{A_{n,N+1}}\Big)\]
and $\Phi(0)=1$, being $\Phi \in \mathcal{D}_{\sigma}(\mathbb{R}^{2d})$ the function in~\eqref{EqDefinitionPhi} with $\omega(t^{1/\rho}) = O(\sigma(t))$, $t \to \infty$.
Since it is enough to see~\eqref{EqEqualityTauLimitOperator} in $\Sch'_{\omega}(\R^d)$, we show that for each $f,g\in \Sch_\omega(\R^d)$,
\begin{equation}\label{EqEqualityTauLimitOperator1}
\iiint e^{i(x-y)\cdot \xi} \Big( \Phi\Big(\frac{(1-\tau)x + \tau y, \xi}{k}\Big) - 1\Big) a(x,y,\xi) f(y) g(x) dy d\xi dx \to 0
\end{equation}
as $k\to\infty$. We use the following identity to integrate by parts with the ultradifferential operator $G(D)$ associated to the entire function in Proposition~\ref{PropPropertiesUltradiffOperator}:
\begin{equation}\label{IntPartsY}
e^{i(x-y)\cdot \xi} = \frac1{G^s(\xi)} G^s(-D_y) e^{i(x-y)\cdot \xi},
\end{equation}
for some power $s \in \mathbb{N}$ that we will determine later. Then, the integrand in the left-hand side of~\eqref{EqEqualityTauLimitOperator1} equals
\begin{align*}
& e^{i(x-y)\cdot \xi} \frac1{G^s(\xi)} G^s(D_y)\Big( \Big( \Phi\Big(\frac{(1-\tau)x + \tau y, \xi}{k}\Big) - 1 \Big) a(x,y,\xi)  f(y) g(x) \Big) \\
& \qquad = e^{i(x-y)\cdot \xi} \frac1{G^s(\xi)} \sum_{\eta \in \mathbb{N}_0^d} b_{\eta} \sum_{\eta_1 + \eta_2 + \eta_3 = \eta} \frac{\eta!}{\eta_1! \eta_2! \eta_3!} \big(\frac{\tau}{k}\big)^{|\eta_1|} D^{\eta_1}_y \Big( \Phi\Big(\frac{(1-\tau)x + \tau y, \xi}{k}\Big) - 1 \Big) \times \\
& \qquad \quad \times D^{\eta_2}_y a(x,y,\xi) D^{\eta_3}_y f(y) g(x).
\end{align*}
Therefore, the integral in~\eqref{EqEqualityTauLimitOperator1} is equal to
\begin{align*}
& \sum_{\eta \in \mathbb{N}_0^d} b_{\eta} \sum_{\eta_1+\eta_2+\eta_3=\eta} \frac{\eta!}{\eta_1! \eta_2! \eta_3!} \big(\frac{\tau}{k}\big)^{|\eta_1|} \iiint e^{i(x-y)\cdot \xi}
\frac1{G^s(\xi)} \times \\
& \quad \times D^{\eta_1}_y\Big( \Phi\Big(\frac{(1-\tau)x + \tau y, \xi}{k}\Big) - 1 \Big) D^{\eta_2}_y a(x,y,\xi)  D^{\eta_3}_y f(y) g(x) dy d\xi dx.
\end{align*}
From Proposition~\ref{PropPropertiesUltradiffOperator}, there are $C_1,C_2,C_3>0$ (depending only on $G$) such that for all $\eta\in\N_0^d$ and $\xi\in\R^d$ we have
\begin{equation}\label{EqEst7}
|b_{\eta}| \leq e^{sC_1} e^{-sC_1 \varphi^{\ast}\big(\frac{|\eta|}{sC_1}\big)}, \qquad \Big|\frac1{G^s(\xi)}\Big| \leq C_3^s e^{-sC_2\omega(\xi)}.
\end{equation}
It follows from Definition~\ref{ga} (see for example~\cite[Lemma 2.6]{AJ}) that for all $\lambda>0$ there exists $C_{\lambda}>0$ such that ($L\geq 1$ is the constant of Lemma~\ref{lemmafistrella})
\[ |D^{\eta_2}_y a(x,y,\xi)| \leq C_{\lambda} e^{\lambda L^3 \varphi^{\ast}\big(\frac{|\eta_2|}{\lambda L^3}\big)} e^{m\omega(x,y,\xi)}. \]
Since $f,g \in \mathcal{S}_{\omega}(\mathbb{R}^d)$, there exist $C'_{\lambda,m}, C_m>0$ such that
\[ |D^{\eta_3}_y f(y)| \leq C'_{\lambda, m} e^{\lambda L^3 \varphi^{\ast}\big(\frac{|\eta_3|}{\lambda L^3 }\big)} e^{-(mL+1) \omega(y)}, \qquad |g(x)|
\leq C_m e^{-(mL+1)\omega(x)}. \]
For $\eta_1=0$ we have $\Phi \equiv 1$ if $|((1-\tau)x + \tau y, \xi)| \leq 2k$, and for $|\eta_1|>0$ it follows that
$D^{\eta_1}_y\Big( \Phi\Big(\frac{(1-\tau)x + \tau y, \xi}{k}\Big) - 1 \Big) = D^{\eta_1}_y \Phi \Big(\frac{(1-\tau)x + \tau y, \xi}{k}\Big)$ is zero for
$|((1-\tau)x + \tau y, \xi)| \leq 2k$; therefore, we can assume that $|((1-\tau)x + \tau y, \xi)| > 2k$. In particular, we have
\[ 1 \leq \frac1{2k} |((1-\tau)x + \tau y, \xi)| \leq \frac1{k}(|1-\tau|+|\tau|)(|x|+1)(|y|+1)(|\xi|+1). \]
As $\Phi \in \mathcal{D}_{\sigma}(\mathbb{R}^{2d}) \subseteq \mathcal{D}_{\omega}(\mathbb{R}^{2d})$, there exists $C''_{\lambda}>0$ such that
\[ |\tau|^{|\eta_1|} \Big| D^{\eta_1}_y \Big(\Phi\Big( \frac{(1-\tau)x + \tau y, \xi}{k}\Big) - 1 \Big) \Big| \leq C''_{\lambda} e^{\lambda L^3
\varphi^{\ast}\big(\frac{|\eta_1|}{\lambda L^3}\big)}, \qquad \eta_1 \in \mathbb{N}_0^d. \]
For $m\geq 0$ (if $m<0$, then $m\omega(x,y,\xi) < 0$), since
\[ m\omega(x,y,\xi) \leq mL \omega(x) + mL \omega(y) + mL\omega(\xi) + mL, \]
it is enough to take $s \in \mathbb{N}$ satisfying $sC_2 \geq mL+1$ to get $e^{(-sC_2 + mL)\omega(\xi)} \leq e^{-\omega(\xi)}$, and therefore the integrals are convergent by condition $(\gamma)$ of the weight $\omega$. On the other hand, since $\sum \frac{\eta!}{\eta_1! \eta_2! \eta_3!} = 3^{|\eta|} \leq e^{2|\eta|}$, by Lemma~\ref{lemmafistrella} we have
\[ \sum_{\eta_1+\eta_2+\eta_3=\eta} \frac{\eta!}{\eta_1! \eta_2! \eta_3!} e^{\lambda L^3 \varphi^{\ast}\big(\frac{|\eta_1|}{\lambda L^3}\big)}
e^{\lambda L^3 \varphi^{\ast}\big(\frac{|\eta_2|}{\lambda L^3}\big)} e^{\lambda L^3 \varphi^{\ast}\big(\frac{|\eta_3|}{\lambda L^3}\big)} \leq
e^{\lambda L \varphi^{\ast}\big(\frac{|\eta|}{\lambda L}\big)} e^{\lambda L^2 + \lambda L^3}. \]
Now, the series
\[ \sum_{\eta \in \mathbb{N}_0^d} e^{-sC_1 \varphi^{\ast}\big(\frac{|\eta|}{sC_1}\big)} e^{\lambda L \varphi^{\ast}\big(\frac{|\eta|}{\lambda L}\big)} \]
converges provided $\lambda > sC_1$ (see~\cite[(3.5), (3.6)]{AJ}). Thus, there exists $C>0$ such that
\[ \Big| \iiint e^{i(x-y)\cdot \xi} \Big( \Phi\Big(\frac{(1-\tau)x + \tau y, \xi}{k}\Big) - 1\Big) a(x,y,\xi) f(y) g(x) dy d\xi dx \Big| \leq C
\frac1{k} \to 0, \]
and hence~\eqref{EqEqualityTauLimitOperator1} is satisfied.
\end{proof}
The next result is the corresponding extension of \cite[Proposition 4.8]{AJ}.
\begin{lema}\label{LemmaSymbolRevisited}
Let $\sum p_j \in \FGS^{m,\omega}_{\rho}$ and let $(C_n)_n, (C'_n)_n$ be the sequences of constants that appear in Definition~\ref{DefFormalSums} and in the estimate of the derivatives of $\varphi_j$ in Corollary~\ref{CorSymbolAmplitudeTau}. We denote $D_n:=C_{2nL^{\widetilde{p}+1}}$ and $D'_n:=C'_{nL^{\widetilde{p}+1}}$, where $L\geq 1$ is the constant of Lemma~\ref{lemmafistrella} and $\widetilde{p}\in\mathbb{N}_0$ is so that $3\langle\tau\rangle\leq e^{\widetilde{p}}$, for a fixed $\tau\in\R$. Consider $(j_n)_n$, $j_n \in \mathbb{N}$, such that $j_1=1$, $j_n < j_{n+1}$, $j_n/n \to \infty$ and
\[ D_{n+1} D'_{n+1} \sum_{j=j_{n+1}}^{\infty} (2R)^{-\rho j} \leq \frac{D_n D'_n}{2} \sum_{j=j_n}^{j_{n+1}-1} (2R)^{-\rho j},\quad n\in\N,$$ and moreover,  $$\frac{n}{j} \varphi^{\ast}\big(\frac{j}{n}\big) \geq \max\{n, \log D_n, \log D'_n \}, \quad \mbox{for }j \geq j_n. \]
If
\[ a(x,\xi) := \sum_{j=0}^{\infty} \varphi_j(x,\xi) p_j(x,\xi), \]
then the associated pseudodifferential operator $A$ is the limit in $L(\mathcal{S}_{\omega}(\mathbb{R}^d),\mathcal{S}'_{\omega}(\mathbb{R}^d))$ of the sequence of operators $S_{N,\tau}:\mathcal{S}_{\omega}(\mathbb{R}^d) \to \mathcal{S}_{\omega}(\mathbb{R}^d)$, where each $S_{N,\tau}$ is a pseudodifferential operator with amplitude
\[ \sum_{j=0}^N (\varphi_j-\varphi_{j+1})((1-\tau)x+\tau y,\xi) (\sum_{l=0}^j p_l((1-\tau)x+\tau y, \xi)).\]
\end{lema}
\begin{proof}
For each $j \in \mathbb{N}_0$, one can show that
\[ (\varphi_j - \varphi_{j+1})((1-\tau)x+\tau y, \xi) \sum_{l=0}^j p_l((1-\tau)x + \tau y, \xi) = \sum_{l=0}^j ((\varphi_j - \varphi_{j+1}) p_l)((1-\tau)x + \tau y, \xi) \]
is a global amplitude in $\GA^{\max\{0,m'\},\omega}_{\rho}$, $m'$ being set in~\eqref{EqDefinitionmprime}. 
Hence, the function
\[ \sum_{j=0}^N (\varphi_j - \varphi_{j+1}) \Big( \sum_{l=0}^j p_l \Big) = \sum_{j=0}^N \varphi_j p_j - \varphi_{N+1} \sum_{l=0}^N p_l \]
is a global amplitude in $\GA^{\max\{0,m'\},\omega}_{\rho}$.

Now, we prove that $S_{N,\tau} \to A$ in $L(\mathcal{S}_{\omega}(\mathbb{R}^d), \mathcal{S}'_{\omega}(\mathbb{R}^d))$ as $N \to \infty$. As in the proof of \cite[Proposition 4.8]{AJ}, it is enough to show that, for any $f,g \in \mathcal{S}_{\omega}(\mathbb{R}^d)$, $\langle (S_{N,\tau} - A)f, g \rangle \to 0$ as $N \to \infty$. Note that $A$ and $S_{N,\tau}$, $N=1,2,\ldots$ act continuously on $\mathcal{S}_{\omega}(\mathbb{R}^d)$. Thus 
\begin{align*}
\langle (S_{N,\tau} - A)f, g \rangle &= \int (S_{N,\tau} - A)f(x) g(x) dx \\
&= \int \Big( \iint e^{i(x-y)\cdot \xi} \Big( \Big\{ \sum_{j=0}^N \varphi_j p_j - \varphi_{N+1} \sum_{l=0}^N p_l \Big\} - a \Big) f(y) dy d\xi \Big) g(x) dx
\end{align*}
for every $f,g \in \mathcal{S}_{\omega}(\mathbb{R}^d)$, where $\varphi_j, \varphi_N, p_j, p_l$, and $a$ are evaluated at $((1-\tau)x+\tau y, \xi)$.

We  show that, for each $f,g \in \mathcal{S}_{\omega}(\mathbb{R}^d)$,
\begin{align*}
a)& \int \Big( \iint e^{i(x-y)\cdot \xi} \Big( \sum_{j=N+1}^{\infty} \varphi_j((1-\tau)x+\tau y,\xi) p_j((1-\tau)x+\tau y,\xi)\Big) f(y) dy d\xi \Big) g(x) dx \quad \mbox{and} \\
b)& \int \Big( \iint e^{i(x-y)\cdot \xi} \Big( \varphi_{N+1}((1-\tau)x+\tau y,\xi) \sum_{l=0}^N p_l((1-\tau)x+\tau y,\xi) \Big) f(y) dy d\xi \Big) g(x) dx
\end{align*}
tend to zero when $N \to \infty$.

Let us show that the integral in $a)$ goes to zero. We 
integrate by parts with  formula~\eqref{IntPartsY} for some $s \in \mathbb{N}$ to be determined later.
Then
\begin{align*}
& e^{i(x-y)\cdot \xi} \frac1{G^s(\xi)} G^s(D_y) \Big( \sum_{j=N+1}^{\infty} \varphi_j \cdot p_j \cdot f(y) \Big) \\
& \qquad = e^{i(x-y)\cdot \xi} \frac1{G^s(\xi)} \sum_{\eta \in \mathbb{N}_0^d} b_{\eta} \sum_{\eta_1+\eta_2+\eta_3=\eta} \frac{\eta!}{\eta_1! \eta_2! \eta_3!} \sum_{j=N+1}^{\infty} \tau^{|\eta_1+\eta_2|} D^{\eta_1}_y \varphi_j \cdot D^{\eta_2}_y p_j \cdot D^{\eta_3}_y f(y).
\end{align*}
Hence, we can reformulate the integral in $a)$ as
\begin{equation}\label{Eq1}
\begin{split}
& \int \Big( \int \frac1{G^s(\xi)} \sum_{\eta \in \mathbb{N}_0^d} b_{\eta} \sum_{\eta_1+\eta_2+\eta_3=\eta} \frac{\eta!}{\eta_1! \eta_2! \eta_3!} \tau^{|\eta_1+\eta_2|} \times \\
& \quad \times \int e^{i(x-y)\cdot \xi} \sum_{j=N+1}^{\infty} D^{\eta_1}_y \varphi_j \cdot D^{\eta_2}_y p_j \cdot D^{\eta_3}_y f(y) dy d\xi \Big) g(x) dx. 
\end{split}
\end{equation}
When $\varphi_j \neq 0$, and $j_n \leq j < j_{n+1}$, we have $\log\big(\frac{\langle((1-\tau)x+\tau y, \xi)\rangle}{2R}\big) \geq \frac{n}{j}\varphi^{\ast}\big(\frac{j}{n}\big)$ (see~\eqref{psi-function}). 
By Corollary~\ref{CorSymbolAmplitudeTau}, for each $n \in \mathbb{N}$, the following estimate holds (as in the hypotheses of this lemma, we denote $D'_n=C'_{nL^{\widetilde{p}+1}}>0$)
\begin{equation*}
|D^{\eta_1}_y \varphi_j((1-\tau)x+\tau y, \xi)| \leq D'_n e^{nL^{\widetilde{p}+1}\varphi^{\ast}\big(\frac{|\eta_1|}{nL^{\widetilde{p}+1}}\big)}.
\end{equation*}
Moreover, for that $n \in \mathbb{N}$ (as in the hypotheses of this lemma, we denote $D_n=C_{2nL^{\widetilde{p}+1}}>0$), by~\eqref{EqEstimationVarphi}, we have
\begin{align*}
& |D^{\eta_2}_y p_j((1-\tau)x+\tau y,\xi)| \\
& \qquad \leq D_n e^{2n L^{\widetilde{p}+1}\rho \varphi^{\ast}\big(\frac{|\eta_2|+j}{2n L^{\widetilde{p}+1}}\big)} \langle((1-\tau)x+\tau y,\xi)\rangle^{-\rho(|\eta_2|+j)} e^{m\omega((1-\tau)x+\tau y,\xi)} \\
& \qquad \leq D_n e^{n L^{\widetilde{p}+1} \varphi^{\ast}\big(\frac{|\eta_2|}{nL^{\widetilde{p}+1}}\big)} e^{nL^{\widetilde{p}+1}\rho \varphi^{\ast}\big(\frac{j}{nL^{\widetilde{p}+1}}\big)} \langle((1-\tau)x+\tau y,\xi)\rangle^{-\rho j} e^{m\omega((1-\tau)x+\tau y, \xi)} \\ 
& \qquad \leq D_n e^{nL^{\widetilde{p}+1} \varphi^{\ast}\big(\frac{|\eta_2|}{nL^{\widetilde{p}+1}}\big)} (2R)^{-\rho j} e^{m\omega((1-\tau)x+\tau y, \xi)}.
\end{align*}
Property $(\gamma)$ of Definition~\ref{DefinitionWeightFunction} yields that there exists $C>0$ such that $\langle x \rangle \leq Ce^{\omega(\langle x \rangle)}$, $x \in \mathbb{R}^d$.  
Then, using~\eqref{EqOmegaX},
\begin{align*}
e^{m\omega((1-\tau)x+\tau y,\xi)} &\leq e^{(m+3)\omega(\langle ( (1-\tau)x+\tau y, \xi) \rangle)} e^{-3\omega(\langle( (1-\tau)x + \tau y, \xi)\rangle)} \\
&\leq e^{(m+3)L \omega( (1-\tau)x+\tau y, \xi)} e^{(m+3)L} C^3 \langle ((1-\tau)x+\tau y, \xi) \rangle^{-3} \\
&\leq C^3 e^{(m+3)L \omega( (1-\tau)x+\tau y, \xi)} e^{(m+3)L} e^{-3\frac{n}{j}\varphi^{\ast}(\frac{j}{n})}.
\end{align*}
By~\eqref{EqWeightTau} ($k$ being as in~\eqref{EqDefinitionk}) and as in~\eqref{EqOmegaDiff}, we obtain 
\begin{align*}
e^{(m+3)L\omega((1-\tau)x+\tau y, \xi)} &\leq e^{(m+3)L^{k+1}\omega(x,y,\xi)} e^{(m+3)L^{k+1} + \cdots + (m+3)L^2} \\ 
&\leq e^{(m+3)L^{k+2}(\omega(x)+\omega(y)+\omega(\xi))} e^{(m+3)L^{k+2} + \cdots + (m+3)L^2}.
\end{align*}

Take $0 < \ell < n$. Later, an additional condition will be imposed on $\ell$. Since $f,g \in \mathcal{S}_{\omega}(\mathbb{R}^d)$, there are $C''_{\ell}>0$, 
which depends on $\ell$,$m$, and on $\tau$, and $D>0$ that depends on $m$ and on $\tau$ such that
\begin{align*}
|D^{\eta_3}_y f(y)| &\leq C''_{\ell} e^{\ell L^{\widetilde{p}+1}\varphi^{\ast}\big(\frac{|\eta_3|}{\ell L^{\widetilde{p}+1}}\big)} e^{-((m+3)L^{k+2}+1)\omega(y)}; \\ 
|g(x)| &\leq D e^{-((m+3)L^{k+2}+1)\omega(x)}.
\end{align*}
Lemma~\ref{lemmafistrella}, the choice of $\widetilde{p}\in \mathbb{N}$  and the fact that $\sum \frac{\eta!}{\eta_1! \eta_2! \eta_3!}= 3^{|\eta|}$ provide
\begin{align*}
& \sum_{\eta_1+\eta_2+\eta_3=\eta} \frac{\eta!}{\eta_1! \eta_2! \eta_3!} |\tau|^{|\eta_1+\eta_2|} e^{nL^{\widetilde{p}+1}\varphi^{\ast}\big(\frac{|\eta_1|}{nL^{\widetilde{p}+1}}\big)} e^{nL^{\widetilde{p}+1}\varphi^{\ast}\big(\frac{|\eta_2|}{nL^{\widetilde{p}+1}}\big)} e^{\ell L^{\widetilde{p}+1}\varphi^{\ast}\big(\frac{|\eta_3|}{\ell L^{\widetilde{p}+1}}\big)} \\
& \qquad \leq \langle \tau \rangle^{|\eta|} e^{\ell L^{\widetilde{p}+1}\varphi^{\ast}\big(\frac{|\eta|}{\ell L^{\widetilde{p}+1}}\big)} \sum_{\eta_1+\eta_2+\eta_3=\eta} \frac{\eta!}{\eta_1! \eta_2! \eta_3!} \\
& \qquad \leq e^{\ell L\varphi^{\ast}\big(\frac{|\eta|}{\ell L}\big)} e^{\ell L\sum_{r=1}^{\widetilde{p}} L^r}.
\end{align*}
Thus, from~\eqref{EqEst7}, we estimate~\eqref{Eq1} by
\begin{align*}
& \int \Big( \int C_3^s e^{-sC_2\omega(\xi)} \sum_{\eta \in \mathbb{N}_0^d} e^{sC_1} e^{-sC_1\varphi^{\ast}\big(\frac{|\eta|}{sC_1}\big)} \Big( \int \sum_{j=N+1}^{\infty} D_n D'_n e^{\ell L\varphi^{\ast}\big(\frac{|\eta|}{\ell L}\big)} e^{\ell L\sum_{r=1}^{\widetilde{p}}L^r} \times \\
& \quad \times (2R)^{-\rho j} C^3 e^{(m+3)L+(m+3)L^2+\cdots+(m+3)L^{k+2}} e^{(m+3)L^{k+2}(\omega(x)+\omega(y)+\omega(\xi))} \times \\
& \quad \times e^{-3\frac{n}{j} \varphi^{\ast}(\frac{j}{n})}  C''_{\ell} e^{-((m+3)L^{k+2}+1)\omega(y)} dy \Big) d\xi \Big) D e^{-((m+3)L^{k+2}+1)\omega(x)} dx.
\end{align*}
Take $s \in \mathbb{N}_0$ such that $sC_2 \geq (m+3)L^{k+2}+1$. 
Choose $\ell \geq sC_1$ so that the series depending on $\eta \in \mathbb{N}_0^d$
\begin{equation*}
\sum_{\eta \in \mathbb{N}_0^d} e^{-sC_1\varphi^{\ast}\big(\frac{|\eta|}{sC_1}\big)} e^{\ell L \varphi^{\ast}\big(\frac{|\eta|}{\ell L}\big)}
\end{equation*}
converges (see~\cite[(3.6)]{AJ}). 
The constant depending on $n$ is $D_nD'_n$.
We get for $j_l \leq N+1 < j_{l+1}$, the following estimate for the integral in $a)$:
\[ E_\ell \Big( \int e^{-\omega(x)} dx \Big) \Big( \int e^{-\omega(y)} dy \Big) \Big( \int e^{-\omega(\xi)} d\xi \Big) \Big( \sum_{n=l}^{\infty} \sum_{j=j_n}^{j_{n+1}-1} \frac{D_n D'_n}{(2R)^{\rho j} e^{3\frac{n}{j}\varphi^{\ast}(\frac{j}{n})}}\Big), \]
where $E_{\ell}>0$ is a constant depending on $\ell$. 
The last 3 integrals converge by property ($\gamma$) of the weight function. By assumption, we have $3\frac{n}{j}\varphi^{\ast}\big(\frac{j}{n}\big) \geq \log D_n + \log D'_n + n$. This finally proves that the integral in $a)$ converges to zero as $N$ tends to infinity.

For the limit in $b)$, we can proceed as in~\cite[Proposition 4.8]{AJ} with the above techniques.
\end{proof}

The next example recovers~\cite[Example 4.9]{AJ} for $\tau=0$. The proof is straightforward and is left to the reader.
\begin{exam}\label{ExampleFormalSumTau}{\rm
	Let $a(x,y,\xi)$ be an amplitude in $\GA^{m,\omega}_{\rho}$ and let
	\[ p_j(x,\xi):= \sum_{|\beta+\gamma|=j} \frac1{\beta!\gamma!} \tau^{|\beta|} (1-\tau)^{|\gamma|} \partial^{\beta+\gamma}_{\xi} (-D_x)^{\beta} D^{\gamma}_y \left. a(x,y,\xi) \right|_{y=x}. \]
	Then, the series $\sum_{j=0}^{\infty} p_j(x,\xi)$ is a formal sum in $\FGS^{\max\{m,mL\},\omega}_{\rho}$ for all $\tau \in \mathbb{R}$.}
\end{exam}

The following lemma is taken from~\cite[Lemma 3.11]{FGJ2005pseudo}.
\begin{lema}\label{LemmaPrevioLarguisimo}
Let $m \geq n$ and $\frac1{e} e^{\frac{m}{j} \varphi^{\ast}(\frac{j}{m})} \leq t \leq e^{\frac{n}{j} \varphi^{\ast}(\frac{j}{n})}$ for $t>0$. Then 
\[ e^{n \varphi^{\ast}(\frac{j}{n})} \geq e^{(n-1)\omega(t)} e^{2n \varphi^{\ast}(\frac{j}{2n})}, \]
for $j$ large enough.
\end{lema}

These two lemmas are easy to prove.
\begin{lema}\label{NotaInequality1}
Let $\tau \in \mathbb{R}$ and let $k \in \mathbb{N}_0$ as in~\eqref{EqDefinitionk}. Then we have
\[ \omega(x,y) \leq L^2\omega((1-\tau)x + \tau y) + L^{k+2}\omega(y-x) + \sum_{j=1}^{k+2}L^j, \qquad x,y \in \mathbb{R}^d. \]
\end{lema}

\begin{lema}\label{NotaInequality2}
For all $\tau \in \mathbb{R}$, the inequality
\[ |v|^2 \leq C(|v+t\tau w|^2 + |v-t(1-\tau)w|^2) \]
holds for all $v,w \in \mathbb{R}^d$, $0\leq t \leq 1$, where $C=2\max\{(1-\tau)^2,\tau^2\}$.
\end{lema}

The following  result shows that any pseudodifferential operator can be written as a quantization modulo an $\omega$-regularizing operator and is needed to understand the composition of two different quatizations in the next section. For the proof, it is fundamental the fact that the kernel $K$ of a pseudodifferential operator behaves like a function in $\Sch_\omega(\R^d)$ in the complement of a strip $\Delta_r = \{ (x,y) \in \mathbb{R}^{2d} : |x-y| < r \}$ around the diagonal of $\R^{2d}$, for some $r>0$. In other words, if $\chi \in \mathcal{E}_{\omega}(\mathbb{R}^{2d})$ is such that $\chi(x,y)=1$ for $(x,y) \in \mathbb{R}^{2d} \setminus \Delta_{2r}$ and $\chi(x,y)= 0$ for $(x,y) \in \overline{\Delta_r}$, then $\chi K\in\mathcal{S}_{\omega}(\mathbb{R}^{2d})$~\cite[Lemma 5.1, Theorem 5.2]{AJ}.

\begin{theo}\label{TheoLarguisimoTau}
Let $a(x,y,\xi)$ be an amplitude in $\GA^{m,\omega}_{\rho}$ with associated pseudodifferential operator $A$. Then, for any $\tau \in \mathbb{R}$, we can write $A$ uniquely as
\[ A = P_{\tau} + R, \]
where $R$ is an $\omega$-regularizing operator and $P_{\tau}$ is the pseudodifferential operator given by
\[ P_{\tau}u(x) = \iint e^{i(x-y)\cdot \xi} p_{\tau}((1-\tau)x+\tau y, \xi) u(y) dy d\xi, \qquad u \in \mathcal{S}_{\omega}(\mathbb{R}^d), \]
being $p_{\tau} \in \GS^{\max\{m,mL\},\omega}_{\rho}$. Moreover, we have
\[ p_{\tau}(x,\xi) \sim \sum_{j=0}^{\infty} p_j(x,\xi) = \sum_{j=0}^{\infty} \sum_{|\beta+\gamma|=j} \frac1{\beta!\gamma!} \tau^{|\beta|} (1-\tau)^{|\gamma|} \partial^{\beta+\gamma}_{\xi} (-D_x)^{\beta} D^{\gamma}_y \left. a(x,y,\xi)\right|_{y=x}.\]
\end{theo}
The symbol $p_{\tau}(x,\xi)$ is called \emph{$\tau$-symbol of the pseudodifferential operator $A$}. When $\tau=0,1,1/2$, these symbols are called the \emph{left, right, and Weyl symbols of $A$}.
\begin{proof}
We consider the sequence $(j_n)_n$ as in the statement of Lemma~\ref{LemmaSymbolRevisited}, with $\frac{n}{j}\varphi^{\ast}\big(\frac{j}{n}\big) \geq \max\{n, \log(C_{4nL^{\widetilde{p}+3}}), \log(D_{4nL^{\widetilde{p}+3}}) \}$, where $(C_n)_n$ and $(D_n)_n$ denote the sequences of constants that come from Definition~\ref{ga} and Corollary~\ref{CorSymbolAmplitudeTau}, and $\widetilde{p}\in\mathbb{N}_0$ is so that $\max\{|1-\tau|,2|\tau|\} \leq e^{\widetilde{p}}$. Put
\[ p_j(x,\xi) := \sum_{|\beta+\gamma|=j} \frac1{\beta!\gamma!} \tau^{|\beta|} (1-\tau)^{|\gamma|} \partial^{\beta+\gamma}_{\xi} (-D_x)^{\beta} D^{\gamma}_y \left. a(x,y,\xi) \right|_{y=x}. \]
By Example~\ref{ExampleFormalSumTau}, $\sum_j p_j \in \FGS^{\max\{m, mL\},\omega}_{\rho}$. Now, we write
\[ p_{\tau}(x,\xi) := \sum_{j=0}^{\infty} \varphi_j(x,\xi) p_j(x,\xi), \]
where $(\varphi_j)_j$ is the sequence described in~\eqref{psi-function}. By~\cite[Theorem 4.6]{AJ} we obtain that $p_{\tau}(x,\xi) \in \GS^{\max\{m, mL\},\omega}_{\rho}$ and $p_{\tau} \sim \sum p_j$. We set, for $u \in \mathcal{S}_{\omega}(\mathbb{R}^d)$,
\[ P_{\tau}u(x) := \iint e^{i(x-y)\cdot \xi} p_{\tau}((1-\tau)x + \tau y, \xi) u(y) dy d\xi. \]
By Lemma~\ref{LemmaSymbolRevisited}, $P_{\tau}$ is the limit of $S_{N,\tau}$ in $L(\mathcal{S}_{\omega}(\mathbb{R}^d),\mathcal{S}'_{\omega}(\mathbb{R}^d))$, where $S_{N,\tau}$ is the pseudodifferential operator with amplitude $\sum_{j=0}^N (\varphi_j-\varphi_{j+1})((1-\tau)x+\tau y,\xi)(\sum_{l=0}^j p_l((1-\tau)x+\tau y,\xi))$ in $\GA^{ \max\{0,m'L\}, \omega}_{\rho}$, $m'$ as in~\eqref{EqDefinitionmprime}. On the other hand, from Lemma~\ref{LemmaAmplitudeRevisited}, $A=\sum_{N=0}^{\infty} A_N$, where $A_N$ is the pseudodifferential operator with amplitude $a(x,y,\xi) (\varphi_N - \varphi_{N+1})((1-\tau)x+\tau y,\xi)$ in $\GA^{m,\omega}_{\rho} \subseteq \GA^{ \max\{0, m'L\}, \omega}_{\rho}$. Thus, for $u \in \mathcal{S}_{\omega}(\mathbb{R}^d)$,
\[ Au(x) = \sum_{N=0}^{\infty} \iint e^{i(x-y)\cdot \xi} (\varphi_N-\varphi_{N+1})((1-\tau)x+\tau y,\xi) a(x,y,\xi) u(y) dy d\xi \]
and
\[ P_{\tau}u(x) = \lim_{N \to \infty} \iint e^{i(x-y)\cdot \xi} \big[ \sum_{j=0}^N (\varphi_j - \varphi_{j+1})((1-\tau)x+\tau y,\xi) \big( \sum_{l=0}^j p_l((1-\tau)x+\tau y, \xi) \big)\big] u(y) dy d\xi. \]
Hence, we can write $A-P_{\tau}$ as the series $\sum_{N=0}^{\infty} P_{N,\tau}$, where each $P_{N,\tau}$ corresponds to the pseudodifferential operator associated to the amplitude in $\GA^{ \max\{0,m'L\}, \omega}_{\rho}$:
\[ (\varphi_N - \varphi_{N+1})((1-\tau)x+\tau y,\xi) \Big( a(x,y,\xi) - \sum_{j=0}^N p_j((1-\tau)x+\tau y,\xi) \Big). \]
Our purpose is to show that the kernel $K$ of $A-P_{\tau}$ belongs to $\mathcal{S}_{\omega}(\mathbb{R}^{2d})$. To that purpose, we write
\begin{align*}
& K(x,y)=\sum_{N=0}^{\infty} K_N(x,y) \\
& \quad = \sum_{N=0}^{\infty} \int e^{i(x-y)\cdot \xi} (\varphi_N - \varphi_{N+1})((1-\tau)x+\tau y,\xi) \Big( a(x,y,\xi) - \sum_{j=0}^N p_j((1-\tau)x+\tau y,\xi) \Big) d\xi.
\end{align*}
Now, we take $r>0$ and $\chi \in \mathcal{E}_{\omega}(\mathbb{R}^{2d})$ such that $\chi \equiv 1$ in $\mathbb{R}^{2d} \setminus \Delta_{2r}$, and $\chi \equiv 0$ in $\overline{\Delta_r}$ (see~\cite[Lemma 5.1]{AJ}). Then we can write
\[ K = \chi K + (1-\chi) \lim_{N \to \infty} \sum_{j=0}^N K_j. \]

We follow the lines of \cite[Theorem 23.2]{Shu2001pseudo}, and also the scheme of the proof of \cite[Theorem~3.13]{FGJ2005pseudo}, as well as \cite[Theorem 5.4]{AJ}. We make the following change of variables:
\begin{equation*}
v = (1-\tau)x + \tau y; \qquad  w = x-y.
\end{equation*}
Similarly as in \cite[Theorem 3.13]{FGJ2005pseudo}, we write the partial sums of $K$ as
\[ \sum_{j=0}^N K_j =K_0 + \sum_{j=1}^N I_j + \sum_{j=1}^N Q_j - W_N, \]
where
\begin{align}
\nonumber I_j(x,y) &:= \sum_{|\beta+\gamma|=j} \sum_{0 \neq \alpha \leq \beta+\gamma} \frac{(\beta+\gamma)!}{\beta! \gamma!} \frac1{\alpha! (\beta+\gamma-\alpha)!} \times \\ \nonumber
& \quad \times \int e^{i(x-y)\cdot \xi} \tau^{|\beta|} (1-\tau)^{|\gamma|} (-1)^{|\gamma|} D^{\alpha}_{\xi} \varphi_j (v,\xi) \big( \partial^{\beta}_x \partial^{\gamma}_y D^{\beta+\gamma-\alpha}_{\xi} a\big)(v,v,\xi) d\xi; \\ \nonumber
Q_j(x,y) &:= \sum_{|\beta+\gamma|=j+1} \sum_{\alpha \leq \beta+\gamma} \frac{(\beta+\gamma)!}{\beta! \gamma!} \frac1{\alpha! (\beta+\gamma-\alpha)!} \tau^{|\beta|} (1-\tau)^{|\gamma|} (-1)^{|\gamma|}\times \\
\nonumber& \quad \times \int e^{i(x-y)\cdot \xi} D^{\alpha}_{\xi}(\varphi_j - \varphi_{j+1})(v,\xi) D^{\beta+\gamma-\alpha}_{\xi} \omega_{\beta \gamma}(x,y,\xi) d\xi; \\ \label{EqDefOmegaBetaGamma}
\omega_{\beta \gamma}(x,y,\xi) &:= (N+1) \int_0^1 \big( \partial^{\beta}_x \partial^{\gamma}_y a\big)(v+t\tau w, v-(1-\tau)tw, \xi) (1-t)^N dt; \\ \nonumber
W_N(x,y) &:= \sum_{|\beta+\gamma|=1}^N \sum_{0 \neq \alpha \leq \beta+\gamma} \frac{(\beta+\gamma)!}{\beta! \gamma!} \frac1{\alpha! (\beta+\gamma-\alpha)!} \times \\ \nonumber
& \quad \times \int e^{i(x-y)\cdot \xi} \tau^{|\beta|} (1-\tau)^{|\gamma|} (-1)^{|\gamma|} D^{\alpha}_{\xi} \varphi_{N+1}(v,\xi) \big( \partial^{\beta}_x \partial^{\gamma}_y D^{\beta+\gamma-\alpha}_{\xi} a\big)(v,v,\xi) d\xi.
\end{align}
We have $\chi K \in \mathcal{S}_{\omega}(\mathbb{R}^{2d})$~\cite[Lemma 5.1, Theorem 5.2]{AJ}. Moreover, it is easy to see that $K_0 \in \mathcal{S}_{\omega}(\mathbb{R}^{2d})$. Indeed, we have
\[ K_0(x,y) = \int e^{i(x-y)\cdot \xi} (1-\varphi_1)((1-\tau)x+\tau y,\xi) (a(x,y,\xi) - a(x,x,\xi)) d\xi. \]
Since $1-\varphi_1 \in \mathcal{S}_{\omega}(\mathbb{R}^{2d})$, following~\cite[Lemma 3.5(a)]{AJ} one obtains the desired property for $K_0$ by Lemma~\ref{NotaInequality1}.

\emph{\underline{First step.}} First of all, we 
compute $D^{\theta}_x D^{\epsilon}_y I_j(x,y)$ for $\theta,\epsilon \in \mathbb{N}_0^d$.
We use integration by parts with the formula
\begin{equation}\label{EqIntegrationByParts1}
e^{i(x-y)\cdot \xi} = \frac1{G(y-x)} G(-D_{\xi}) e^{i(x-y)\cdot \xi},
\end{equation}
for a suitable power $G^s(D)$ of $G(D)$, being $G(\xi)$ the entire function considered in Proposition~\ref{PropPropertiesUltradiffOperator}. 
We obtain, as in~\cite[Theorem 5.4]{AJ},
\begin{align*}
& D^{\theta}_x D^{\epsilon}_y I_j(x,y) \\
& \qquad = \sum_{|\beta+\gamma|=j} \sum_{0 \neq \alpha \leq \beta+\gamma} \frac{(\beta+\gamma)!}{\beta! \gamma!} \frac1{\alpha! (\beta+\gamma-\alpha)!} \frac1{G^s(y-x)} \sum_{\eta \in \mathbb{N}_0^d} b_{\eta} \sum_{\substack{\eta_1+\eta_2+\eta_3 = \eta \\ \theta_1+\theta_2+\theta_3 = \theta \\ \epsilon_1+\epsilon_2+\epsilon_3 = \epsilon}} (-1)^{|\gamma+\epsilon_1|} \times \\
& \quad \qquad \times \frac{\eta!}{\eta_1! \eta_2! \eta_3!} \frac{\theta!}{\theta_1! \theta_2! \theta_3!} \frac{\epsilon!}{\epsilon_1! \epsilon_2! \epsilon_3!} \frac{(\theta_1+\epsilon_1)!}{(\theta_1+\epsilon_1-\eta_1)!} \tau^{|\beta+\epsilon_2|} (1-\tau)^{|\gamma+\theta_2|} \times \\
& \quad \qquad \times \int e^{i(x-y)\cdot \xi} \xi^{\theta_1+\epsilon_1-\eta_1} D^{\theta_2}_x D^{\epsilon_2}_y D^{\alpha+\eta_2}_{\xi} \varphi_j(v,\xi) D^{\theta_3}_x D^{\epsilon_3}_y \big( \partial^{\beta}_x \partial^{\gamma}_y D^{\beta+\gamma-\alpha+\eta_3}_{\xi} a\big)(v,v,\xi) d\xi.
\end{align*}

Fix $\lambda>0$ and set $n\geq \lambda$ large enough that may depend on $\tau,m,\rho,L$, and $R$.
According to Lemma~\ref{NotaInequality1}, it is enough to take $s \in \mathbb{N}$ such that $sC_2 \geq \lambda L^{k+2}$, where $C_2>0$ comes from~\eqref{EqEst7} and $k \in \mathbb{N}_0$ as in~\eqref{EqDefinitionk}. 
For the convergence of the series depending on $\eta \in \mathbb{N}_0^d$, 
let $n$ satisfy in addition that $n \geq sC_1$, where $C_1>0$ comes from~\eqref{EqEst7}. Now, proceeding as in \cite[Theorem 3.13]{FGJ2005pseudo} (and using Proposition~\ref{PropTechnical} and Lemma~\ref{LemmaPrevioLarguisimo}) we can show that $\sum_{j=1}^{\infty} I_j \in \mathcal{S}_{\omega}(\mathbb{R}^{2d})$.


\emph{\underline{Second step.}} 
Since $1-\chi$ is supported in $\Delta_{2r}$, we estimate $|D^{\theta}_x D^{\epsilon}_y Q_j(x,y)|$ for $\theta, \epsilon \in \mathbb{N}_0^d$, $(x,y) \in \Delta_{2r}$. 
By the formula of integration by parts given in~\eqref{EqIntegrationByParts1} for a suitable power 
of $G(D)$, $G^s(D)$, we have
\begin{align*}
& D^{\theta}_x D^{\epsilon}_y Q_j(x,y) \\
& \qquad = \sum_{|\beta+\gamma|=j+1} \sum_{\alpha \leq \beta+\gamma} \frac{(\beta+\gamma)!}{\beta! \gamma!} \frac1{\alpha! (\beta+\gamma-\alpha)!} \frac1{G^s(y-x)} \sum_{\eta \in \mathbb{N}_0^d} b_{\eta} \sum_{\substack{\theta_1+\theta_2+\theta_3=\theta \\ \epsilon_1+\epsilon_2+\epsilon_3=\epsilon \\ \eta_1+\eta_2+\eta_3=\eta}} (-1)^{|\epsilon_1+\gamma|} \times \\
& \qquad \quad \times \frac{\theta!}{\theta_1! \theta_2! \theta_3!} \frac{\epsilon!}{\epsilon_1! \epsilon_2! \epsilon_3!} \frac{\eta!}{\eta_1! \eta_2! \eta_3!} \frac{(\theta_1+\epsilon_1)!}{(\theta_1+\epsilon_1-\eta_1)!} \tau^{|\beta|} (1-\tau)^{|\gamma|} (1-\tau)^{|\theta_2|} \tau^{|\epsilon_2|} \times \\
& \qquad \quad \times \int e^{i(x-y)\cdot \xi} \xi^{\theta_1+\epsilon_1-\eta_1} D^{\theta_2}_x D^{\epsilon_2}_y D^{\alpha+\eta_2}_{\xi} (\varphi_j - \varphi_{j+1})(v,\xi) D^{\theta_3}_x D^{\epsilon_3}_y (D^{\beta+\gamma-\alpha+\eta_3}_{\xi} \omega_{\beta \gamma}) d\xi,
\end{align*}
where $\omega_{\beta \gamma} = \omega_{\beta \gamma}(x,y,\xi)$ is defined in~\eqref{EqDefOmegaBetaGamma}.
Fix $\lambda>0$ and take $n \geq \lambda$ to be determined later. We consider in this step $\widetilde{p} \in \mathbb{N}$ such that
\[ \max\{ 2(1+|\tau|), (1+2r)^{\rho} \} \leq e^{\rho\widetilde{p}}. \]
We put $\widetilde{n} \in \mathbb{N}_0$, $\widetilde{n} \geq n$, such that (where $q \in \mathbb{N}_0$ satisfies $2^q \geq 3R$)
\[ \widetilde{n} \geq \frac{L^{q+1}}{\rho}(\lambda L^{\widetilde{p}+2}+mL^3+1)+1. \]
By Lemma~\ref{NotaInequality2} and the properties of $\varphi^{\ast}$, proceeding as in the proof of the second step of \cite[Theorem 5.4]{AJ} we obtain, for some $C_{\widetilde{n}}>0$,
\begin{align*}
& |D^{\theta_3}_x D^{\epsilon_3}_y (D^{\beta+\gamma-\alpha+\eta_3}_{\xi} \omega_{\beta \gamma})(x,y,\xi)| \\
& \quad \leq C_{\widetilde{n}} e^{16\widetilde{n} L^{\widetilde{p}+3} \rho \sum_{p=1}^{3\widetilde{p}+1} L^p} e^{mL^{k+3}+\cdots+mL} (j+1) \langle(v,\xi)\rangle^{-\rho|2\beta+2\gamma-\alpha|} \times \\
& \quad \quad \times e^{16\widetilde{n} L^{\widetilde{p}+3} \rho \varphi^{\ast}\big(\frac{|2\beta+2\gamma-\alpha+\theta_3+\epsilon_3+\eta_3|}{16\widetilde{n}L^{\widetilde{p}+3}}\big)} e^{mL^3\omega(v)} e^{mL^{k+3}\omega(w)} e^{mL\omega(\xi)} \int_0^1 |1-t|^j dt.
\end{align*}
For the estimate of the derivatives of $Q_j(x,\xi)$ we 
can proceed similarly as in the first step to show finally that $(1-\chi)\sum_{j=1}^{\infty} Q_j \in \mathcal{S}_{\omega}(\mathbb{R}^{2d})$. 

\emph{\underline{Third step.}}
Let $T_N: \mathcal{S}_{\omega}(\mathbb{R}^d) \to \mathcal{S}_{\omega}(\mathbb{R}^d)$ be the operator with kernel $(1-\chi)W_N$. As in the proof of~\cite[Theorem 5.4]{AJ}, 
it follows that $(T_N)$ converges to an operator $T:\mathcal{S}_{\omega}(\mathbb{R}^d) \to \mathcal{S}_{\omega}(\mathbb{R}^d)$ in $L(\mathcal{S}_{\omega}(\mathbb{R}^d), \mathcal{S}'_{\omega}(\mathbb{R}^d))$.
We show that
$T=0$. 
To this aim, fix $N \in \mathbb{N}$, $j_n \leq N+1 < j_{n+1}$ 
and set $a_N := Re^{\frac{n}{N+1}\varphi^{\ast}\big(\frac{N+1}{n}\big)}$. For the support of the derivatives of $\varphi_{N+1}$, we may assume that
\[ 2a_N \leq \langle((1-\tau)x+\tau y,\xi)\rangle \leq 3a_N. \]
For $f,g \in \mathcal{S}_{\omega}(\mathbb{R}^d)$, we have
\[ \langle T_Nf, g \rangle = \int T_Nf(x) g(x) dx = \int \Big( \int (1-\chi)(x,y) W_N(x,y) f(y) dy \Big) g(x) dx. \]
Fixed $N \in \mathbb{N}$, we can use Fubini's theorem (since $f,g \in \mathcal{S}_{\omega}(\mathbb{R}^d)$ and $|\xi| \leq 3a_N$) and we obtain
\begin{align*}
\langle T_Nf, g \rangle &= \int \Big( \int \sum_{|\beta+\gamma|=1}^N \sum_{0 \neq \alpha \leq \beta+\gamma} \frac{(\beta+\gamma)!}{\beta! \gamma!} \frac1{\alpha! (\beta+\gamma-\alpha)!} \Big\{ \int e^{i(x-y)\cdot \xi} \tau^{|\beta|} (1-\tau)^{|\gamma|} (-1)^{|\gamma|} \times \\
& \quad \times D^{\alpha}_{\xi} \varphi_{N+1}(v,\xi) \big( \partial^{\beta}_x \partial^{\gamma}_y D^{\beta+\gamma-\alpha}_{\xi} a\big)(v,v,\xi) d\xi \Big\} f(y) (1-\chi)(x,y) dy \Big) g(x) dx.
\end{align*}

An integration by parts with~\eqref{EqIntegrationByParts1} for a suitable power $s \in \mathbb{N}$, to be determined, gives
\begin{align*}
& e^{i(x-y)\cdot \xi} \frac1{G^s(\xi)} G^s(D_y) \big\{ D^{\alpha}_{\xi} \varphi_{N+1}(v,\xi) \big( \partial^{\beta}_x \partial^{\gamma}_y D^{\beta+\gamma-\alpha}_{\xi} a \big)(v,v,\xi) f(y) (1-\chi)(x,y) g(x) \big\} \\
& \quad = e^{i(x-y)\cdot \xi} \frac1{G^s(\xi)} \sum_{\eta \in \mathbb{N}_0^d} b_{\eta} \sum_{\eta_1+\eta_2+\eta_3+\eta_4=\eta} \frac{\eta!}{\eta_1! \eta_2! \eta_3! \eta_4!} \tau^{|\eta_1|} D^{\eta_1}_y D^{\alpha}_{\xi} \varphi_{N+1}(v,\xi) \times \\
& \quad \quad \times D^{\eta_2}_y \big( \partial^{\beta}_x \partial^{\gamma}_y D^{\beta+\gamma-\alpha}_{\xi} a\big)(v,v,\xi) D^{\eta_3}_y f(y) D^{\eta_4}_y (1-\chi)(x,y) g(x).
\end{align*}
Thus, we obtain
\begin{align*}
\langle T_Nf, g \rangle &= \sum_{|\beta+\gamma|=1}^N \sum_{0 \neq \alpha \leq \beta+\gamma} \frac{(\beta+\gamma)!}{\beta! \gamma!} \frac1{\alpha! (\beta+\gamma-\alpha)!} \sum_{\eta \in \mathbb{N}_0^d} b_{\eta} \sum_{\eta_1+\eta_2+\eta_3+\eta_4=\eta} \frac{\eta!}{\eta_1! \eta_2! \eta_3! \eta_4!} \times \\
& \quad \times \tau^{|\eta_1+\beta|} (1-\tau)^{|\gamma|} (-1)^{|\gamma|} \int \! \int e^{i(x-y)\cdot \xi} \frac1{G^s(\xi)} \int D^{\eta_1}_y D^{\alpha}_{\xi} \varphi_{N+1}(v,\xi) \times \\
& \quad \times D^{\eta_2}_y \big( \partial^{\beta}_x \partial^{\gamma}_y D^{\beta+\gamma-\alpha}_{\xi} a\big)(v,v,\xi) D^{\eta_3}_y f(y) D^{\eta_4}_y (1-\chi)(x,y) g(x) dy  d\xi dx.
\end{align*}
To estimate $|\langle T_Nf, g\rangle|$, let $\widetilde{p} \in \mathbb{N}_0$ be as at the beginning of the proof ($\max\{ |1-\tau|, 2|\tau| \} \leq e^{\widetilde{p}}$). By Definition~\ref{ga} and Corollary~\ref{CorSymbolAmplitudeTau}, for all $n \in \mathbb{N}$ there exist $C_n=C_{4nL^{\widetilde{p}+3}}>0$ and $D_n=D_{4nL^{\widetilde{p}+3}}>0$ such that, by the chain rule,
\begin{align*}
& |D^{\eta_2}_y (\partial^{\beta}_x \partial^{\gamma}_y D^{\beta+\gamma-\alpha}_{\xi} a)(v,v,\xi)| \\
& \qquad \leq 2^{|\eta_2|} |\tau|^{|\eta_2|} C_n \langle(v,\xi)\rangle^{-\rho|2\beta+2\gamma+\eta_2-\alpha|} e^{4n L^{\widetilde{p}+3}\rho \varphi^{\ast}\big(\frac{|2\beta+2\gamma+\eta_2-\alpha|}{4nL^{\widetilde{p}+3}}\big)} e^{m\omega(v,v,\xi)}, \\
& |D^{\eta_1}_y D^{\alpha}_{\xi} \varphi_{N+1}(v,\xi)| \leq D_n \langle(v,\xi)\rangle^{-\rho|\eta_1+\alpha|} e^{4n L^{\widetilde{p}+3} \rho \varphi^{\ast}\big(\frac{|\eta_1+\alpha|}{4nL^{\widetilde{p}+3}}\big)}.
\end{align*}
By the choice of $\widetilde{p} \in \mathbb{N}_0$,
\[ |\tau|^{|\eta_1+\beta|} (2|\tau|)^{|\eta_2|} |1-\tau|^{|\gamma|} \leq e^{\widetilde{p}|\eta_1+\eta_2+\beta+\gamma|}. \]
Since $2a_N \leq \langle(v,\xi)\rangle$ and $1\leq|\beta+\gamma|\leq N<N+1$, we use that $\varphi^{\ast}(x)/x$ is increasing to get
\begin{align*}
\langle(v,\xi)\rangle^{-\rho|\eta_1+\alpha|} \langle(v,\xi)\rangle^{-\rho|2\beta+2\gamma+\eta_2-\alpha|} &\leq \langle(v,\xi)\rangle^{-\rho|2\beta+2\gamma|} \\
&\leq (2R)^{-2\rho|\beta+\gamma|} e^{-2n\rho \varphi^{\ast}\big(\frac{|\beta+\gamma|}{n}\big)}.
\end{align*}

Put $\ell < n$. Since $f,g \in \mathcal{S}_{\omega}(\mathbb{R}^d)$ and $1-\chi \in \mathcal{E}_{\omega}(\mathbb{R}^{2d})$, there exist $E_{\ell}, E'_{\ell}, E>0$ such that (where $k$ is as in~\eqref{EqDefinitionk})
\begin{align*}
|D^{\eta_3}_y f(y)| &\leq E_{\ell} e^{\ell L^3\varphi^{\ast}\big(\frac{|\eta_3|}{\ell L^3}\big)} e^{-((mL+L)L^{k+1}+1)\omega(y)}; \\
|D^{\eta_4}_y (1-\chi)(x,y)| &\leq E'_{\ell} e^{\ell L^3\varphi^{\ast}\big(\frac{|\eta_4|}{\ell L^3}\big)}; \\
|g(x)| &\leq E' e^{-((mL+L)L^{k+1}+1)\omega(x)}.
\end{align*}
We use~\eqref{EqEst7}. Since $\frac{(\beta+\gamma)!}{\beta!\gamma!} \leq 2^{|\beta+\gamma|} \leq e^{|\beta+\gamma|}$, we have by the properties of $\varphi^{\ast}$ that $|\langle T_Nf, g \rangle|$ is less than or equal to
\begin{align*}
& \sum_{|\beta+\gamma|=1}^N \sum_{0 \neq \alpha \leq \beta+\gamma} \Big(\frac{e^{\widetilde{p}+1}}{(2R)^{2\rho}}\Big)^{|\beta+\gamma|} \frac1{\alpha! (\beta+\gamma-\alpha)!} e^{sC_1}  \Big( \sum_{\eta \in \mathbb{N}_0^d} e^{\ell L \varphi^{\ast}\big(\frac{|\eta|}{\ell L}\big)} e^{-sC_1 \varphi^{\ast}\big(\frac{|\eta|}{sC_1}\big)}\Big) e^{\ell L \sum_{t=1}^{\widetilde{p}+2} L^t} \times \\
& \quad \times \int \Big( \int C_3^s e^{-sC_2\omega(\xi)} \Big( \int C_n D_n E_{\ell} E'_{\ell} E' e^{m\omega(v,v,\xi)} e^{-((mL+L)L^{k+1}+1)(\omega(y)+\omega(x))} dy \Big) d\xi \Big) dx.
\end{align*}
Set $s \in \mathbb{N}$ such that $sC_2 \geq (mL+L)L^{k+1}+1$, and take $\ell \geq sC_1$ to get that the series is convergent. It is easy to see for such $s \in \mathbb{N}$ that there exists $C_k>0$ such that
\begin{align*}
& e^{m\omega(v,v,\xi)} e^{-((mL+L)L^{k+1}+1)\omega(y)} e^{-((mL+L)L^{k+1}+1)\omega(x)} e^{-sC_2\omega(\xi)} \\
& \quad \leq C_k e^{-\omega(\langle(v,\xi)\rangle)} e^{-\omega(x)} e^{-\omega(y)} e^{-\omega(\xi)}.
\end{align*}
So, we have
\begin{align*}
& \iiint_{2a_N \leq \langle(v,\xi)\rangle \leq 3a_N} e^{-\omega(\langle(v,\xi)\rangle)} e^{-\omega(x)-\omega(y)-\omega(\xi)} dy d\xi dx \\
& \quad \leq e^{-\omega(2a_N)} \iiint_{\mathbb{R}^{3d}} e^{-\omega(x)-\omega(y)-\omega(\xi)} dy d\xi dx,
\end{align*}
By property $(\gamma)$ of Definition~\ref{DefinitionWeightFunction}, there exists $C>0$ such that $3\log(t) \leq \omega(t)+C$, $t\geq0$. Thus,
\[ e^{-\omega(2a_N)} \leq (2a_N)^{-3} e^C. \]
We recall that $C_nD_n$ is the only constant that depends on $n$. By the choice of the sequence $(j_n)_n$, we have
\[ e^n C_n D_n \leq a_N^3. \]
Hence, there exists $C'>0$ such that 
\begin{align*}
|\langle T_Nf, g \rangle| &\leq C' \sum_{|\beta+\gamma|=1}^N \sum_{0\neq \alpha \leq \beta+\gamma} \Big( \frac{e^{\widetilde{p}+1}}{(2R)^{2\rho}} \Big)^{|\beta+\gamma|} \frac1{\alpha! (\beta+\gamma-\alpha)!} \frac{C_nD_n}{a_N^3} \\
&\leq \frac{C'}{e^n} \sum_{l=1}^N \frac1{l!} \Big( \frac{de^{\widetilde{p}+1}}{(2R)^{2\rho}} \Big)^l.
\end{align*}
Since the series converges for $R\geq1$ large enough (which may depend on $\tau$), and since $n\to \infty$ when $N\to\infty$, we show that $|\langle T_Nf,g\rangle|$ tends to zero when $N\to \infty$.

It only remains to prove the uniqueness of the pseudodifferential operator modulo an $\omega$-regularizing operator. We notice that every global amplitude as in Definition~\ref{ga} defines an $\omega$-ultradistribution. Then, as in \cite{NR2010global,Shu2001pseudo}, 
the identities in $\Sch'_\omega(\R^{2d})$ for the Fourier transform
\begin{equation*}
K_{\tau}(x,y) = (2\pi)^d \mathcal{F}^{-1}_{\xi \mapsto x-y} \big( a_{\tau}((1-\tau)x+\tau y, \xi) \big)
\end{equation*}
and
\begin{equation*}
a_{\tau}(v,\xi) = (2\pi)^{-d} \mathcal{F}_{w \mapsto \xi} \big( K_{\tau}(v+\tau w, v-(1-\tau)w) \big)
\end{equation*}
yield the uniqueness of the $\tau$-symbol since  the kernel $K_{\tau}$ is also unique.
\end{proof}

As a consequence of Theorem~\ref{TheoLarguisimoTau}, we can describe the precise relation between different quantizations for a given global symbol in terms of equivalence of formal sums as the following result shows.
\begin{theo}\label{TheoRelationDifferentSymbols}
If $a_{\tau_1}(x,\xi)$ and $a_{\tau_2}(x,\xi)$ are the $\tau_1$ and $\tau_2$-symbol of the same pseudodifferential operator $A$, then
\[ a_{\tau_2}(x,\xi) \sim \sum_{j=0}^{\infty} \sum_{|\alpha|=j} \frac1{\alpha!} (\tau_1-\tau_2)^{|\alpha|} \partial^{\alpha}_{\xi} D^{\alpha}_x a_{\tau_1}(x,\xi).\]
\end{theo}
\begin{proof}
By Theorem~\ref{TheoLarguisimoTau}, the pseudodifferential operator $A$ is determined via the $\tau_1$-symbol $a_{\tau_1}((1-\tau_1)x+\tau_1 y,\xi)$ modulo an $\omega$-regularizing operator. 
Again by Theorem~\ref{TheoLarguisimoTau}, its $\tau_2$-symbol has the following asymptotic expansion
\begin{align*}
a_{\tau_2}(x,\xi) &\sim \sum_{j=0}^{\infty} \sum_{|\beta+\gamma|=j} \frac{(-1)^{|\beta|}}{\beta! \gamma!} \tau_2^{|\beta|} (1-\tau_2)^{|\gamma|} \partial^{\beta+\gamma}_{\xi} D^{\beta}_x D^{\gamma}_y (a_{\tau_1}((1-\tau_1)x+\tau_1 y,\xi)\big|_{y=x}) \\
&= \sum_{j=0}^{\infty} \sum_{|\alpha|=j} \Big( \sum_{\beta+\gamma=\alpha} \frac1{\beta! \gamma!} ((1-\tau_2)\tau_1)^{|\gamma|} (-\tau_2(1-\tau_1))^{|\beta|}\Big) \partial^{\alpha}_{\xi} D^{\alpha}_x a_{\tau_1}(x,\xi) \\
&= \sum_{j=0}^{\infty} \sum_{|\alpha|=j} \frac1{\alpha!} ((1-\tau_2)\tau_1 - \tau_2(1-\tau_1))^{|\alpha|} \partial^{\alpha}_{\xi} D^{\alpha}_x a_{\tau_1}(x,\xi) \\
&= \sum_{j=0}^{\infty} \sum_{|\alpha|=j} \frac1{\alpha!} (\tau_1-\tau_2)^{|\alpha|} \partial^{\alpha}_{\xi} D^{\alpha}_x a_{\tau_1}(x,\xi).
\end{align*}
\end{proof}



\section{Transposition and composition of operators}\label{SectionTransposeAndComposition}
By~\cite[Proposition 3.10]{AJ}, we deduce that if $A$ has as amplitude $a((1-\tau)x+\tau y,\xi)$, then its transpose $^tA$ has the amplitude $a((1-\tau)y+\tau x, -\xi)$. Hence, if $a_{\tau}(x,\xi)$ is the $\tau$-symbol of $A$, then $^ta_{1-\tau}(x,\xi)$ is the $(1-\tau)$-symbol of $^tA$ given by
\begin{equation}\label{EqTranspose1}
^ta_{1-\tau}((1-\tau)x + \tau y, \xi) := a_{\tau}((1-\tau)y + \tau x, -\xi).
\end{equation}
In particular 
we have $^ta_{\tau}(x,\xi) = a_{1-\tau}(x,-\xi)$. On the other hand, for $\tau=0$, $^ta_{1}(y,-\xi)$ coincides with $a_0(x,\xi)$. Now, we show the corresponding generalization of~\cite[Proposition 5.5]{AJ}.
\begin{theo}\label{TheoTransposeTau}
Let $A$ be the pseudodifferential operator with $\tau$-symbol $a_{\tau}(x,\xi)$. Then its transpose  restricted to $\mathcal{S}_{\omega}(\mathbb{R}^d)$ can be decomposed as $^tA=Q+R$, where $R$ is an $\omega$-regularizing operator and $Q$ is the pseudodifferential operator associated to the $\tau$-symbol given by
\[ q(x,\xi) \sim \sum_{j=0}^{\infty} \sum_{|\alpha|=j} \frac1{\alpha!} (1-2\tau)^{|\alpha|} \partial^{\alpha}_{\xi} D^{\alpha}_x a_{\tau}(x,-\xi). \]
\end{theo}
\begin{proof}
By assumption we deduce that $^tA$ has the $(1-\tau)$-symbol $^ta_{1-\tau}(x,\xi)$ given by formula~\eqref{EqTranspose1} 
restricted to $y=x$. Moreover, from Theorem~\ref{TheoRelationDifferentSymbols}, the $\tau$-symbol of $^tA$ satisfies
\[ ^ta_{\tau}(x,\xi) \sim \sum_{j=0}^{\infty} \sum_{|\alpha|=j} \frac1{\alpha!} (1-2\tau)^{|\alpha|} \partial^{\alpha}_{\xi} D^{\alpha t}_x a_{1-\tau}(x,\xi) = \sum_{j=0}^{\infty} \sum_{|\alpha|=j} \frac1{\alpha!} (1-2\tau)^{|\alpha|} \partial^{\alpha}_{\xi} D^{\alpha}_x a_{\tau}(x,-\xi). \]
\vspace{-0.5mm}
\end{proof}

Let us deal with the composition of two pseudodifferential operators given by their corresponding quantizations of symbols.
\begin{theo}\label{TheoCompositionTau}
Let $a_{\tau_1}(x,\xi) \in \GS^{m_1,\omega}_{\rho}$ be the $\tau_1$-symbol of $A_1$ and $b_{\tau_2}(x,\xi) \in \GS^{m_2,\omega}_{\rho}$ be the $\tau_2$-symbol of $A_2$, being $A_1$ and $A_2$ their corresponding pseudodifferential operators. The $\tau$-symbol $c_{\tau}(x,\xi) \in \GS^{m_1+m_2,\omega}_{\rho}$ of $A_1 \circ A_2$ has the asymptotic expansion
\begin{equation}\label{EqCoefficientsComposition}
\sum_{j=0}^{\infty} \sum_{\substack{|\alpha+\beta-\alpha_1-\alpha_2| = j \\ \alpha+\beta=\gamma+\delta}} c_{\alpha \beta \gamma \delta \alpha_1 \alpha_2} \partial^{\gamma}_{\xi} D^{\alpha}_x a_{\tau_1}(x,\xi) \cdot \partial^{\delta}_{\xi} D^{\beta}_x b_{\tau_2}(x,\xi),
\end{equation}
where the coefficients $c_{\alpha \beta \gamma \delta \alpha_1 \alpha_2}$ 
are
\[ \frac{(2\pi)^d}{\gamma! \delta!} \! \sum_{k,l=0}^{\infty} \sum_{\substack{|\alpha_1|=k \\ |\alpha_2|=l}} \! (-1)^{|\alpha-\alpha_1+\alpha_2|} \binom{\alpha+\beta-\alpha_1-\alpha_2}{\alpha-\alpha_1} \binom{\gamma}{\alpha_1} \binom{\delta}{\alpha_2} \tau^{|\alpha-\alpha_1|} (1-\tau)^{|\beta-\alpha_2|} \tau_1^{|\alpha_1|} (1-\tau_2)^{|\alpha_2|}. \]
\end{theo}
\begin{proof}
We first assume $\tau_1=0$ and $\tau_2=1$. In this case, $a_{\tau_1}((1-\tau_1)x+\tau_1 y,\xi)$ and $b_{\tau_2}((1-\tau_2)x+\tau_2 y,\xi)$ coincide with $ a_0(x,\xi)$ and $ b_1(y,\xi)$. Then
\[ (A_1 \circ A_2)u(x) = \int e^{ix\cdot \xi} a_0(x,\xi) \widehat{A_2u}(\xi) d\xi, \qquad x \in \mathbb{R}^d. \]
It is not difficult to see that $A_2u(x) = \widehat{I}(-x)$, where $I(\xi) = \int e^{-iy\cdot \xi} b_1(y,\xi) u(y) dy$. 
Hence $\widehat{A_2u}(\xi) = (2\pi)^dI(\xi)$ and
\[ (A_1 \circ A_2)u(x) = \iint e^{i(x-y)\cdot \xi} c(x,y,\xi) u(y) dy d\xi, \qquad x \in \mathbb{R}^d, \]
where $c(x,y,\xi) = (2\pi)^d a_0(x,\xi) b_1(y,\xi)$ is an amplitude in $\GA^{m_1+m_2,\omega}_{\rho}$. So, by Theorem~\ref{TheoLarguisimoTau}, the $\tau$-symbol $c_{\tau}(x,\xi)$ has the asymptotic expansion:
\begin{align}\label{EqComp1}
c_{\tau}(x,\xi) &\sim (2\pi)^d \sum_{j=0}^{\infty} \sum_{|\beta+\gamma|=j} \frac{(-1)^{|\beta|}}{\beta! \gamma!} \tau^{|\beta|} (1-\tau)^{|\gamma|} \partial^{\beta+\gamma}_{\xi} D^{\beta}_x D^{\gamma}_y \big( a_0(x,\xi) b_1(y,\xi) \big)\big|_{y=x} \\ \label{EqComp2} 
&= (2\pi)^d \sum_{j=0}^{\infty} \sum_{\substack{|\beta+\gamma|=j \\ \delta+\epsilon=\beta+\gamma}} \frac{(-1)^{|\beta|} (\beta+\gamma)!}{\delta! \epsilon! \beta! \gamma!} \tau^{|\beta|} (1-\tau)^{|\gamma|} \partial^{\delta}_{\xi} D^{\beta}_x a_0(x,\xi) \cdot \partial^{\epsilon}_{\xi} D^{\gamma}_x b_1(x,\xi).
\end{align}

For the general case, by Theorem~\ref{TheoRelationDifferentSymbols}, we have
\begin{align*}
a_0(x,\xi) &\sim \sum_{j_1=0}^{\infty} \sum_{|\alpha_1|=j_1} \frac1{\alpha_1!} \tau_1^{|\alpha_1|} \partial^{\alpha_1}_{\xi} D^{\alpha_1}_x a_{\tau_1}(x,\xi); \\
b_1(x,\xi) &\sim \sum_{j_2=0}^{\infty} \sum_{|\alpha_2|=j_2} \frac{(-1)^{|\alpha_2|}}{\alpha_2!} (1-\tau_2)^{|\alpha_2|} \partial^{\alpha_2}_{\xi} D^{\alpha_2}_x b_{\tau_2}(x,\xi).
\end{align*}
Thus, from~\eqref{EqComp2}, we get
\begin{align*}
c_{\tau}(x,\xi) &\sim (2\pi)^d \sum_{j=0}^{\infty} \sum_{\substack{|\beta+\gamma|=j \\ \delta+\epsilon=\beta+\gamma}} \frac{(-1)^{|\beta|} (\beta+\gamma)!}{\delta! \epsilon! \beta! \gamma!} \tau^{|\beta|} (1-\tau)^{|\gamma|} \times \\
& \quad \times \partial^{\delta}_{\xi} D^{\beta}_x \Big( \sum_{j_1=0}^{\infty} \sum_{|\alpha_1|=j_1} \frac1{\alpha_1!} \tau_1^{|\alpha_1|} \partial^{\alpha_1}_{\xi} D^{\alpha_1}_x a_{\tau_1}(x,\xi) \Big) \times \\
& \quad \times \partial^{\epsilon}_{\xi} D^{\gamma}_x \Big( \sum_{j_2=0}^{\infty} \sum_{|\alpha_2|=j_2} \frac{(-1)^{|\alpha_2|}}{\alpha_2!} (1-\tau_2)^{|\alpha_2|} \partial^{\alpha_2}_{\xi} D^{\alpha_2}_x b_{\tau_2}(x,\xi) \Big).
\end{align*}
We make the change of variables $\gamma' = \alpha_1 + \delta$, $\alpha' = \alpha_1 + \beta$, $\delta' = \alpha_2 + \epsilon$, $\beta' = \alpha_2 + \gamma$.
Then
\begin{align*}
c_{\tau}(x,\xi) &\sim (2\pi)^d \sum_{j=0}^{\infty} \sum_{\substack{|\alpha'+\beta'-\alpha_1-\alpha_2|=j \\ \alpha'+\beta' = \delta'+\gamma'}} \frac1{\gamma'! \delta'!} \partial^{\gamma'}_{\xi} D^{\alpha'}_x a_{\tau_1}(x,\xi) \partial^{\delta'}_{\xi} D^{\beta'}_x b_{\tau_2}(x,\xi) \times \\
& \quad \times \sum_{k,l=0}^{\infty} \sum_{\substack{|\alpha_1|=k \\ |\alpha_2|=l}} (-1)^{|\alpha'-\alpha_1+\alpha_2|} \frac{(\alpha'+\beta'-\alpha_1-\alpha_2)!}{(\alpha'-\alpha_1)! (\beta'-\alpha_2)!} \frac{\gamma'!}{\alpha_1! (\gamma'-\alpha_1)!} \frac{\delta'!}{\alpha_2! (\delta'-\alpha_2)!} \times \\
& \quad \times \tau^{|\alpha'-\alpha_1|} (1-\tau)^{|\beta'-\alpha_2|} \tau_1^{|\alpha_1|} (1-\tau_2)^{|\alpha_2|}.
\end{align*}
The proof follows since
\[ \frac{(\alpha'+\beta'-\alpha_1-\alpha_2)!}{(\alpha'-\alpha_1)! (\beta'-\alpha_2)!} \frac{\gamma'!}{\alpha_1! (\gamma'-\alpha_1)!} \frac{\delta'!}{\alpha_2! (\delta'-\alpha_2)!} = \binom{\alpha'+\beta'-\alpha_1-\alpha_2}{\alpha'-\alpha_1} \binom{\gamma'}{\alpha_1} \binom{\delta'}{\alpha_2}. \]

\end{proof}

The coefficients appearing in formula~\eqref{EqCoefficientsComposition} are sometimes simplified for some particular $\tau \in \mathbb{R}$. For example, if $\tau=0$, by formula~\eqref{EqComp1}, we obtain
\[ c(x,\xi) = c_0(x,\xi) \sim (2\pi)^d \sum_{j=0}^{\infty} \sum_{|\gamma|=j} \frac1{\gamma!} \partial^{\gamma}_{\xi} D^{\gamma}_y \big( a_0(x,\xi) b_1(y,\xi) \big)\big|_{y=x}. \]
On the other hand, from formula~\eqref{EqTranspose1}, $b_1(x,\xi) =$ $^tb_0(x,-\xi)$. Hence, by~\cite[Lemma 5.6]{AJ}, we have
\begin{equation*}
c_0(x,\xi) \sim (2\pi)^d \big( a_0(x,\xi) \circ b_0(x,\xi) \big) = (2\pi)^d \sum_{j=0}^{\infty} \sum_{|\gamma|=j} \frac1{\gamma!} \partial^{\gamma}_{\xi} a_0(x,\xi) D^{\gamma}_x b_0(x,\xi),
\end{equation*}
which in particular gives~\cite[Theorem 5.7]{AJ} (cf.~\cite[Theorem 23.7]{Shu2001pseudo}).

Another interesting case is when dealing with $\tau=1/2$. We will obtain it as a consequence of a more general result (cf.~\cite[Problem 23.2]{Shu2001pseudo}). First, we need a lemma, taken from~\cite[Theorem 5.5]{BBR1996Global}:
\begin{lema}\label{LemmaEqualityTheo5.5BBR}
The formula
\[ \frac{(\beta+\gamma)!}{(\beta+\gamma-\epsilon)! \epsilon!} \frac1{\beta! \gamma!} = \sum_{\substack{0 \leq \delta \leq \beta \\ \beta-\epsilon \leq \delta \leq \beta-\epsilon+\gamma}} \frac1{(\beta-\delta)! (\beta-\epsilon+\gamma-\delta)! \delta! (\delta-\beta+\epsilon)!}, \]
holds for all $\beta,\gamma,\epsilon \in \mathbb{N}_0^d$ with $\epsilon \leq \beta+\gamma$.
\end{lema}

\begin{exam}
Given two pseudodifferential operators $A$ and $B$, the $\tau$-symbol of the composition operator $C=A \circ B$ is given by
\[ c_{\tau}(x,\xi) \sim (2\pi)^d \sum_{j=0}^{\infty} \sum_{|\beta+\gamma|=j} \frac{(-1)^{|\beta|}}{\beta! \gamma!} \tau^{|\beta|} (1-\tau)^{|\gamma|} (\partial^{\gamma}_{\xi} D^{\beta}_x a_{\tau}(x,\xi)) (\partial^{\beta}_{\xi} D^{\gamma}_x b_{\tau}(x,\xi)). \]
\end{exam}
\begin{proof}
Formula~\eqref{EqComp2} states that $c_{\tau}(x,\xi)$ is equivalent to (since $\delta=\beta+\gamma-\epsilon$)
\[ (2\pi)^d \sum_{j=0}^{\infty} \sum_{|\beta+\gamma|=j} (-1)^{|\beta|} \tau^{|\beta|} (1-\tau)^{|\gamma|} \sum_{\epsilon \leq \beta+\gamma} \frac{(\beta+\gamma)!}{(\beta+\gamma-\epsilon)! \epsilon!} \frac1{\beta! \gamma!} \partial^{\beta+\gamma-\epsilon}_{\xi} D^{\beta}_x a_0(x,\xi) \cdot \partial^{\epsilon}_{\xi} D^{\gamma}_x b_1(x,\xi). \]
Moreover, by Lemma~\ref{LemmaEqualityTheo5.5BBR}, it is equal to
\begin{align*}
& (2\pi)^d \sum_{j=0}^{\infty} \sum_{|\beta+\gamma|=j} (-1)^{|\beta|} \tau^{|\beta|} (1-\tau)^{|\gamma|} \times \\
& \quad \times \sum_{\epsilon \leq \beta+\gamma} \sum_{\substack{0 \leq \delta \leq \beta \\ \beta-\epsilon \leq \delta \leq \beta-\epsilon+\gamma}} \frac1{(\beta-\delta)! (\beta-\epsilon+\gamma-\delta)! \delta! (\delta-\beta+\epsilon)!} \partial^{\beta+\gamma-\epsilon}_{\xi} D^{\beta}_x a_0(x,\xi) \cdot \partial^{\epsilon}_{\xi} D^{\gamma}_x b_1(x,\xi).
\end{align*}
We put $\mu = \beta-\delta$, $\nu = \beta-\epsilon+\gamma-\delta$, and $\theta=\delta-\beta+\epsilon$. Therefore,
\begin{align*}
c_{\tau}(x,\xi) &\sim (2\pi)^d \sum_{j=0}^{\infty} \sum_{|\nu+\theta+\mu+\delta|=j} \frac{(-1)^{|\mu+\delta|}}{\mu! \nu! \delta! \theta!} \tau^{|\mu+\delta|} (1-\tau)^{|\nu+\theta|} \times \\
&\quad \times \partial^{\nu+\delta}_{\xi} D^{\mu+\delta}_x a_0(x,\xi) \cdot  \partial^{\theta+\mu}_{\xi} D^{\nu+\theta}_x b_1(x,\xi),
\end{align*}
and taking $j=j_1+j_2+j_3$, $j_1, j_2, j_3 \in \mathbb{N}_0$, we have
\begin{align*}
c_{\tau}(x,\xi) &\sim (2\pi)^d \sum_{j_1=0}^{\infty} \sum_{|\nu + \mu|=j_1} \! \! \frac{(-1)^{|\mu|}}{\mu! \nu!} \tau^{|\mu|} (1-\tau)^{|\nu|} \partial^{\nu}_{\xi} D^{\mu}_x \Big( \sum_{j_2=0}^{\infty} \sum_{|\delta|=j_2} \! \frac{(-1)^{|\delta|}}{\delta!} \tau^{|\delta|} \partial^{\delta}_{\xi} D^{\delta}_x a_0(x,\xi) \Big) \times \\
& \quad \times \partial^{\mu}_{\xi} D^{\nu}_x \Big( \sum_{j_3=0}^{\infty} \sum_{|\theta|=j_3} \frac1{\theta!} (1-\tau)^{|\theta|} \partial^{\theta}_{\xi} D^{\theta}_x b_1(x,\xi) \Big).
\end{align*}
We get the result since Theorem~\ref{TheoRelationDifferentSymbols} gives
\[ a_{\tau}(x,\xi) \sim \sum_{k=0}^{\infty} \sum_{|\delta|=k} \frac{(-1)^{|\delta|}}{\delta!} \tau^{|\delta|} \partial^{\delta}_{\xi} D^{\delta}_x a_0(x,\xi), \qquad
b_{\tau}(x,\xi) \sim \sum_{k=0}^{\infty} \sum_{|\theta|=k} \frac1{\theta!} (1-\tau)^{|\theta|} \partial^{\theta}_{\xi} D^{\theta}_x b_1(x,\xi). \]
\end{proof}

\begin{cor}
Given two pseudodifferential operators $A$ and $B$, the Weyl symbol of the composition operator $C=A \circ B$ is given by
\[ c_w(x,\xi) 
\sim (2\pi)^d \sum_{j=0}^{\infty} \sum_{|\beta+\gamma|=j} \frac{(-1)^{|\beta|}}{\beta! \gamma!} 2^{-|\beta+\gamma|} (\partial^{\gamma}_{\xi} D^{\beta}_x a_w(x,\xi)) (\partial^{\beta}_{\xi} D^{\gamma}_x b_w(x,\xi)). \]
\end{cor}

\section{Parametrices and $\omega$-regularity}\label{Parametrix}
In this section we give a sufficient condition for \emph{$\omega$-regularity} of a global pseudodifferential operator. We say that a pseudodifferential operator $P:\Sch'_\omega(\R^d)\to \Sch'_\omega(\R^d)$ is \emph{$\omega$-regular} if given $u\in \Sch'_\omega(\R^d)$ such that $Pu\in \Sch_\omega(\R^d)$, we have $u\in \Sch_\omega(\R^d)$. See \cite{BJORegularity} for a study of $\omega$-regularity of linear partial differential operators with polynomial coefficients using quadratic transformations (cf.~\cite{MO2020Regularity} for the non-isotropic case).

We use the well-known method of the construction of a parametrix for the symbol of the operator, using symbolic calculus. We follow the lines of \cite{FGJ2004hypo,Z1986pseudodifferential}. From \cite{PV1984almost}, we know that a weight function $\sigma$ is equivalent to a subadditive weight function if and only if it satisfies
\begin{itemize}
\item[$(\alpha_0)$]
\vspace{0.5mm} $\displaystyle \exists C>0, \ \exists t_0>0 \ \forall \lambda \geq 1: \ \sigma(\lambda t) \leq \lambda C \sigma(t), \qquad t \geq t_0$.
\end{itemize}
We refer to \cite{FG2006superposition,PV1984almost} for applications and characterizations of property $(\alpha_0)$ on the weight function.
The following result is taken from \cite[Lemma 3.3]{FGJ2004hypo}.
\begin{lema}\label{Lemma323TesisDavid}
Let $\omega$ be a subadditive weight function. For all $\lambda>0$ and $j,k \in \mathbb{N}$, we have
\[ \frac{e^{\lambda \varphi^{\ast}_{\omega}(\frac{j}{\lambda})}}{j!} \frac{e^{\lambda \varphi^{\ast}_{\omega}(\frac{k}{\lambda})}}{k!} \leq \frac{e^{\lambda \varphi^{\ast}_{\omega}(\frac{j+k}{\lambda})}}{(j+k)!}. \]
\end{lema}

The following lemma states Vandermonde's identity.
\begin{lema}\label{LemmaVandermondeIdentity}
For any $m,n,r \in \mathbb{N}_0$, we have
\[ \sum_{k=0}^r \binom{m}{k} \binom{n}{r-k} = \binom{m+n}{r}. \]
\end{lema}

\begin{lema}\label{Lemma324TesisDavid}
If $\sum_j a_j \in \FGS^{m_1,\omega}_{\rho}$ and $b(x,\xi) \in \GS^{m_2,\omega}_{\rho}$, then $\sum_j a_j(x,\xi)b(x,\xi) \in \FGS^{m_1+m_2,\omega}_{\rho}$.
\end{lema}

The following result is in the spirit of Zanghirati~\cite{Z1986pseudodifferential} and Fern\'andez, Galbis, and Jornet~\cite{FGJ2004hypo} (see also Cappiello, Pilipovi\'c, and Prangoski~\cite{CPP}).
\begin{theo}\label{TheoParametrix}
Let $\omega$ be a weight function and let $\sigma$ be a subadditive weight function with $\omega(t^{1/\rho}) = o(\sigma(t))$ as $t \to \infty$. Let $p(x,\xi) \in \GS^{|m|,\omega}_{\rho}$ be such that, for some $R\geq 1$:
\begin{itemize}
\item[(i)] $\displaystyle |p(x,\xi)| \geq \frac1{R} e^{-|m|\omega(x,\xi)}$ for $\langle(x,\xi)\rangle \geq R$;
\item[(ii)] There exist $C>0$ and $n \in \mathbb{N}$ such that
\[ |D^{\alpha}_x D^{\beta}_{\xi} p(x,\xi)| \leq C^{|\alpha+\beta|} \langle(x,\xi)\rangle^{-\rho|\alpha+\beta|} e^{\frac1{n}\varphi_{\sigma}^{\ast}(n|\alpha|)} e^{\frac1{n} \varphi_{\sigma}^{\ast}(n|\beta|)} |p(x,\xi)|, \]
for $\alpha,\beta \in \mathbb{N}_0^d$, $\langle(x,\xi)\rangle \geq R$.
\end{itemize}
Then there exists $q(x,\xi) \in \GS^{|m|,\omega}_{\rho}$ such that $q \circ p \sim 1$ in $\FGS^{|m|,\omega}_{\rho}$. 
\end{theo}

\begin{proof}
We set
\[ q_0(x,\xi) = \frac1{p(x,\xi)}, \qquad \langle(x,\xi)\rangle \geq R. \]
We show by induction on $|\alpha+\beta|$ that there exists $C_1>0$ such that
\begin{equation}\label{EqTheo326TesisDavid1}
|D^{\alpha}_x D^{\beta}_{\xi} q_0(x,\xi)| \leq C_1^{|\alpha+\beta|} \langle(x,\xi)\rangle^{-\rho|\alpha+\beta|} e^{\frac1{n}\varphi_{\sigma}^{\ast}(n|\alpha|)} e^{\frac1{n}\varphi_{\sigma}^{\ast}(n|\beta|)} |q_0(x,\xi)|
\end{equation}
for all $\alpha,\beta \in \mathbb{N}_0^d$, $\langle(x,\xi)\rangle \geq R$. Indeed, the inequality is true for $\alpha=\beta=0$. Now,  differentiating formula $p(x,\xi)q_0(x,\xi)=1$, we obtain
\[ p(x,\xi) D^{\alpha}_x D^{\beta}_{\xi} q_0(x,\xi) = - \sum_{0 \neq (\widehat{\alpha},\widehat{\beta})\le (\alpha,\beta)} \frac{\alpha!}{\widehat{\alpha}! (\alpha-\widehat{\alpha})!} \frac{\beta!}{\widehat{\beta}! (\beta-\widehat{\beta})!} D^{\widehat{\alpha}}_x D^{\widehat{\beta}}_{\xi} p(x,\xi) D^{\alpha-\widehat{\alpha}}_x D^{\beta-\widehat{\beta}}_{\xi} q_0(x,\xi). \]
Now, we assume that the inequality \eqref{EqTheo326TesisDavid1} is true for $(\widehat{\alpha},\widehat{\beta})<(\alpha,\beta)$. Using condition $(ii)$, we obtain
\begin{align*}
& |p(x,\xi) D^{\alpha}_x D^{\beta}_{\xi} q_0(x,\xi)| \\
& \qquad \leq \sum_{0 \neq (\widehat{\alpha},\widehat{\beta}) \leq (\alpha,\beta)} \frac{\alpha!}{\widehat{\alpha}! (\alpha-\widehat{\alpha})!} \frac{\beta!}{\widehat{\beta}! (\beta-\widehat{\beta})!} C^{|\widehat{\alpha}+\widehat{\beta}|} \langle(x,\xi)\rangle^{-\rho|\widehat{\alpha}+\widehat{\beta}|} e^{\frac1{n}\varphi_{\sigma}^{\ast}(n|\widehat{\alpha}|)} e^{\frac1{n}\varphi_{\sigma}^{\ast}(n|\widehat{\beta}|)} |p(x,\xi)| \times \\
& \qquad \quad \times C_1^{|\alpha-\widehat{\alpha}+\beta-\widehat{\beta}|} \langle(x,\xi)\rangle^{-\rho|\alpha-\widehat{\alpha}+\beta-\widehat{\beta}|} e^{\frac1{n}\varphi_{\sigma}^{\ast}(n|\alpha-\widehat{\alpha}|)} e^{\frac1{n}\varphi_{\sigma}^{\ast}(n|\beta-\widehat{\beta}|)} |q_0(x,\xi)| .
\end{align*}
Since $\frac{\alpha!}{\widehat{\alpha}! (\alpha-\widehat{\alpha})!} \frac{\beta!}{\widehat{\beta}! (\beta-\widehat{\beta})!} \leq \frac{|\alpha|!}{|\widehat{\alpha}|! |\alpha-\widehat{\alpha}|!} \frac{|\beta|!}{|\widehat{\beta}|! |\beta-\widehat{\beta}|!}$, we obtain, by Lemma~\ref{Lemma323TesisDavid},
\[ |\alpha|! \frac{e^{\frac1{n}\varphi_{\sigma}^{\ast}(n|\widehat{\alpha}|)}}{|\widehat{\alpha}|!} \frac{e^{\frac1{n}\varphi_{\sigma}^{\ast}(n|\alpha-\widehat{\alpha}|)}}{|\alpha-\widehat{\alpha}|!} |\beta|! \frac{e^{\frac1{n}\varphi_{\sigma}^{\ast}(n|\widehat{\beta}|)}}{|\widehat{\beta}|!} \frac{e^{\frac1{n}\varphi_{\sigma}^{\ast}(n|\beta-\widehat{\beta}|)}}{|\beta-\widehat{\beta}|!} \leq e^{\frac1{n}\varphi_{\sigma}^{\ast}(n|\alpha|)} e^{\frac1{n}\varphi_{\sigma}^{\ast}(n|\beta|)}. \]
Thus
\[ |D^{\alpha}_x D^{\beta}_{\xi} q_0(x,\xi)| \leq C_1^{|\alpha+\beta|} \langle(x,\xi)\rangle^{-\rho|\alpha+\beta|} e^{\frac1{n}\varphi_{\sigma}^{\ast}(n|\alpha|)} e^{\frac1{n}\varphi_{\sigma}^{\ast}(n|\beta|)} |q_0(x,\xi)| \sum_{0 \neq (\widehat{\alpha},\widehat{\beta}) \leq (\alpha,\beta)} \Big( \frac{C}{C_1} \Big)^{|\widehat{\alpha}+\widehat{\beta}|}. \]
Finally, the fact that
\[ \sum_{0 \neq (\widehat{\alpha},\widehat{\beta})\le (\alpha,\beta)} \Big( \frac{C}{C_1} \Big)^{|\widehat{\alpha}+\widehat{\beta}|}\le \sum_{k=1}^{|\alpha+\beta|} \sum_{|\eta|=k} \Big( \frac{C}{C_1}\Big)^k \leq \sum_{k=1}^{|\alpha+\beta|} \Big( \frac{dC}{C_1}\Big)^k \]
completes the proof of \eqref{EqTheo326TesisDavid1} if we take $C_1>0$ such that
\[ \sum_{k=1}^{\infty} \Big( \frac{dC}{C_1}\Big)^k < 1.\]

For $j \in \mathbb{N}$, we define recursively
\[ q_j(x,\xi) := -q_0(x,\xi) \! \! \! \! \sum_{0 < |\epsilon+\gamma| \leq j} \! \! \! \! \frac{(-1)^{|\epsilon|}}{\epsilon! \gamma!} \tau^{|\epsilon|} (1-\tau)^{|\gamma|} (\partial^{\gamma}_{\xi} D^{\epsilon}_x q_{j-|\epsilon+\gamma|}(x,\xi)) (\partial^{\epsilon}_{\xi} D^{\gamma}_x p(x,\xi)). \]
We show that there exist constants $C_2,C_3>0$ with $C_1 < C_2 < C_3$ such that
\begin{equation}\label{EqTheo326TesisDavid2}
|D^{\alpha}_x D^{\beta}_{\xi} q_j(x,\xi)| \leq C_2^{|\alpha+\beta|} C_3^j \langle(x,\xi)\rangle^{-\rho(|\alpha+\beta|+2j)} e^{\frac1{n}\varphi_{\sigma}^{\ast}(n(|\alpha+\beta|+2j))} e^{|m|\omega(x,\xi)},
\end{equation}
for all $\alpha,\beta \in \mathbb{N}_0^d$, $\langle(x,\xi)\rangle \geq R$. We proceed by induction on $j \in \mathbb{N}_0$. First, observe that formula~\eqref{EqTheo326TesisDavid1} implies formula \eqref{EqTheo326TesisDavid2} for $j=0$, since $|q_0(x,\xi)| \leq Re^{|m|\omega(x,\xi)}$ for $\langle(x,\xi)\rangle \geq R$ (from condition $(i)$). Now, assume that~\eqref{EqTheo326TesisDavid2} holds for all $0 \leq l < j$ (where $C_3>C_2>C_1$, and $C_2,C_3>0$ are large enough). Then, by the definition of $q_j(x,\xi)$, we have
\begin{align*}
|D^{\alpha}_x D^{\beta}_{\xi} q_j(x,\xi)| &\leq \sum_{\substack{\alpha_1+\alpha_2+\alpha_3=\alpha \\ \beta_1+\beta_2+\beta_3=\beta}} \frac{\alpha!}{\alpha_1!\alpha_2!\alpha_3!} \frac{\beta!}{\beta_1!\beta_2!\beta_3!} |D^{\alpha_1}_x D^{\beta_1}_{\xi} q_0(x,\xi)| \! \sum_{0 < |\epsilon+\gamma| \leq j} \! \frac1{\epsilon! \gamma!} \times \\
& \ \ \times |\tau|^{|\epsilon|} |1-\tau|^{|\gamma|} |D^{\alpha_2+\epsilon}_x D^{\beta_2+\gamma}_{\xi} q_{j-|\epsilon+\gamma|}(x,\xi)| |D^{\alpha_3+\gamma}_x D^{\beta_3+\epsilon}_{\xi} p(x,\xi)|.
\end{align*}
We use formula~\eqref{EqTheo326TesisDavid1} for the derivatives of $q_0(x,\xi)$, the inductive hypothesis~\eqref{EqTheo326TesisDavid2} for the ones of $q_{j-|\mu|}(x,\xi)$, and condition $(ii)$ for the derivatives of $p(x,\xi)$. All this implies
\begin{align} \nonumber
|D^{\alpha}_x D^{\beta}_{\xi} q_j(x,\xi)| &\leq \sum_{\substack{\alpha_1+\alpha_2+\alpha_3=\alpha \\ \beta_1+\beta_2+\beta_3=\beta}} \frac{\alpha!}{\alpha_1!\alpha_2!\alpha_3!} \frac{\beta!}{\beta_1!\beta_2!\beta_3!} C_1^{|\alpha_1+\beta_1|} \langle(x,\xi)\rangle^{-\rho|\alpha_1+\beta_1|} e^{\frac1{n}\varphi_{\sigma}^{\ast}(n|\alpha_1|)} \times \\ \nonumber
& \quad \ \ \times e^{\frac1{n}\varphi_{\sigma}^{\ast}(n|\beta_1|)} |q_0(x,\xi)| \sum_{0 < |\epsilon+\gamma| \leq j} \frac1{\epsilon! \gamma!} |\tau|^{|\epsilon|} |1-\tau|^{|\gamma|} C_2^{|\alpha_2+\epsilon+\beta_2+\gamma|} C_3^{j-|\epsilon+\gamma|} \times \\ \nonumber
& \quad \ \ \times \langle(x,\xi)\rangle^{-\rho(|\alpha_2+\epsilon+\beta_2+\gamma|+2(j-|\epsilon+\gamma|))} e^{\frac1{n}\varphi_{\sigma}^{\ast}(n(|\alpha_2+\epsilon+\beta_2+\gamma|+2(j-|\epsilon+\gamma|)))} e^{|m|\omega(x,\xi)} \times \\ \nonumber
& \quad \ \ \times C^{|\alpha_3+\gamma+\beta_3+\epsilon|} \langle(x,\xi)\rangle^{-\rho|\alpha_3+\gamma+\beta_3+\epsilon|} e^{\frac1{n}\varphi_{\sigma}^{\ast}(n|\alpha_3+\gamma|)} e^{\frac1{n}\varphi_{\sigma}^{\ast}(n|\beta_3+\epsilon|)} |p(x,\xi)| \\ \label{q_j}
\begin{split}
&= \langle(x,\xi)\rangle^{-\rho(|\alpha+\beta|+2j)} e^{|m|\omega(x,\xi)} \sum_{\substack{\alpha_1+\alpha_2+\alpha_3=\alpha \\ \beta_1+\beta_2+\beta_3=\beta}} \frac{\alpha!}{\alpha_1!\alpha_2!\alpha_3!} \frac{\beta!}{\beta_1!\beta_2!\beta_3!} C_1^{|\alpha_1+\beta_1|} \times \\
& \quad \ \ \times e^{\frac1{n}\varphi_{\sigma}^{\ast}(n|\alpha_1|)} e^{\frac1{n}\varphi_{\sigma}^{\ast}(n|\beta_1|)} \sum_{0 < |\epsilon+\gamma| \leq j} \frac1{\epsilon! \gamma!} |\tau|^{|\epsilon|} |1-\tau|^{|\gamma|} C_2^{|\alpha_2+\epsilon+\beta_2+\gamma|} C_3^{j-|\epsilon+\gamma|} \times \\
& \quad \ \ \times e^{\frac1{n}\varphi_{\sigma}^{\ast}(n(|\alpha_2+\beta_2|+2j-|\epsilon+\gamma|))} C^{|\alpha_3+\gamma+\beta_3+\epsilon|} e^{\frac1{n}\varphi_{\sigma}^{\ast}(n|\alpha_3+\gamma|)} e^{\frac1{n}\varphi_{\sigma}^{\ast}(n|\beta_3+\epsilon|)}.
\end{split}
\end{align}

To estimate the right-hand side of~\eqref{q_j} we multiply and divide by $$(|\alpha_2+\beta_2|+2j-|\epsilon+\gamma|)! |\alpha_3+\gamma|! |\beta_3+\epsilon|!.$$ Then, as $$\frac{\alpha!}{\alpha_1! \alpha_2! \alpha_3!} \frac{\beta!}{\beta_1! \beta_2! \beta_3!} \leq \frac{|\alpha|!}{|\alpha_1|! |\alpha_2|! |\alpha_3|!} \frac{|\beta|!}{|\beta_1|! |\beta_2|! |\beta_3|!},$$ we have, by Lemma~\ref{Lemma323TesisDavid},
\begin{align*}
& \frac{e^{\frac1{n}\varphi_{\sigma}^{\ast}(n|\alpha_1|)}}{|\alpha_1|!} \frac{e^{\frac1{n}\varphi_{\sigma}^{\ast}(n|\beta_1|)}}{|\beta_1|!} \frac{e^{\frac1{n}\varphi_{\sigma}^{\ast}(n(|\alpha_2+\beta_2|+2j-|\epsilon+\gamma|))}}{(|\alpha_2+\beta_2|+2j-|\epsilon+\gamma|)!} \frac{e^{\frac1{n}\varphi_{\sigma}^{\ast}(n|\alpha_3+\gamma|)}}{|\alpha_3+\gamma|!} \frac{e^{\frac1{n}\varphi_{\sigma}^{\ast}(n|\beta_3+\epsilon|)}}{|\beta_3+\epsilon|!}  \\
& \quad \quad \leq \frac1{(|\alpha+\beta|+2j)!} e^{\frac1{n}\varphi_{\sigma}^{\ast}(n(|\alpha+\beta|+2j))}.
\end{align*}
Now, we see that
\begin{equation}\label{EqTheo326TesisDavid3}
\frac{|\alpha|!}{|\alpha_2|! |\alpha_3|!} \frac{|\beta|!}{|\beta_2|! |\beta_3|!} |\alpha_3+\gamma|! |\beta_3+\epsilon|! \frac{(|\alpha_2+\beta_2|+2j-|\epsilon+\gamma|)!}{(|\alpha+\beta|+2j)!} \leq 2^{|\alpha_1+\alpha_3|} 2^{|\beta_1+\beta_3|}.
\end{equation}
Indeed, we multiply and divide by $(|\alpha_1+\alpha_3| + |\beta_1+\beta_3| + |\epsilon+\gamma|)!$ to get, by the properties of the multinomial coefficients,
\begin{align*}
& \frac{|\alpha|!}{|\alpha_2|! |\alpha_3|!} \frac{|\beta|!}{|\beta_2|! |\beta_3|!} \frac{|\alpha_3+\gamma|! |\beta_3+\epsilon|!}{(|\alpha_1+\alpha_3| + |\beta_1+\beta_3| + |\epsilon+\gamma|)!} \frac1{\binom{|\alpha+\beta|+2j}{|\alpha_2+\beta_2| + 2j - |\epsilon+\gamma|}} \\
& \qquad \leq \frac{|\alpha|!}{|\alpha_2|! |\alpha_3|!} \frac{|\beta|!}{|\beta_2|! |\beta_3|!} \frac1{|\alpha_1|! |\beta_1|!} \frac1{\binom{|\alpha+\beta|+2j}{|\alpha_2+\beta_2| + 2j - |\epsilon+\gamma|}}.
\end{align*}
As we have, for $\alpha=\alpha_1+\alpha_2+\alpha_3$,
\[ \frac{|\alpha|!}{|\alpha_1|! |\alpha_2|! |\alpha_3|!} = \frac{|\alpha_1+\alpha_3|!}{|\alpha_1|! |\alpha_3|!} \binom{|\alpha|}{|\alpha_2|} \leq 2^{|\alpha_1+\alpha_3|} \binom{|\alpha|}{|\alpha_2|}, \]
(and in the same way for $\beta=\beta_1+\beta_2+\beta_3$), we deduce formula~\eqref{EqTheo326TesisDavid3} by Lemma~\ref{LemmaVandermondeIdentity}. We then have, from~\eqref{q_j},
\begin{align*}
|D^{\alpha}_x D^{\beta}_{\xi} q_j(x,\xi)| &\leq \langle(x,\xi)\rangle^{-\rho(|\alpha+\beta|+2j)} e^{\frac1{n} \varphi^{\ast}_{\sigma}(n(|\alpha+\beta|+2j))} e^{|m|\omega(x,\xi)} \times \\
&\quad \times \sum_{\substack{\alpha_1+\alpha_2+\alpha_3=\alpha \\ \beta_1+\beta_2+\beta_3=\beta}} 2^{|\alpha_1+\alpha_3|} 2^{|\beta_1+\beta_3|} C_1^{|\alpha_1+\beta_1|} C_2^{|\alpha_2+\beta_2|} C_3^j C^{|\alpha_3+\beta_3|} \times \\
&\quad \times \sum_{0<|\epsilon+\gamma|\leq j} \frac1{\epsilon! \gamma!} |\tau|^{|\epsilon|} |1-\tau|^{|\gamma|} C_2^{|\epsilon+\gamma|} C_3^{-|\epsilon+\gamma|} C^{|\epsilon+\gamma|}.
\end{align*}

Since
\begin{align*}
C_2^{|\alpha+\beta|} C_3^j \sum_{\substack{\alpha_1+\alpha_2+\alpha_3=\alpha \\ \beta_1+\beta_2+\beta_3=\beta}} \Big(\frac{2C_1}{C_2}\Big)^{|\alpha_1+\beta_1|} \Big(\frac{2C}{C_2}\Big)^{|\alpha_3+\beta_3|} &\leq C_2^{|\alpha+\beta|} C_3^j \sum_{\substack{\alpha_1+\alpha_2+\alpha_3=\alpha \\ \beta_1+\beta_2+\beta_3=\beta}} \Big(\frac{2CC_1}{C_2}\Big)^{|\alpha_1+\alpha_3+\beta_1+\beta_3|} \\
&\leq C_2^{|\alpha+\beta|} C_3^j \sum_{k=0}^{|\alpha+\beta|} \sum_{|\eta|=k} \Big(\frac{2CC_1}{C_2}\Big)^k,
\end{align*}
we take $C_2>0$ large enough so that
\[ \sum_{k=0}^{\infty} \Big( \frac{2dCC_1}{C_2} \Big)^k < 2. \]
In addition, we put $C_3>0$ large enough satisfying
\begin{align*}
\sum_{0<|\epsilon|\leq j} \frac1{\epsilon!} \Big( \frac{CC_2|\tau|}{C_3} \Big)^{|\epsilon|} \! \! \sum_{0<|\gamma| \leq j} \frac1{\gamma!} \Big( \frac{CC_2|1-\tau|}{C_3} \Big)^{|\gamma|} &\leq \Big( \sum_{0<k\leq j} \frac1{k!} \Big( \frac{d^2 CC_2 \max\{|\tau|,|1-\tau|\}}{C_3} \Big)^{k}\Big)^2 \\
&\leq \Big( \sum_{k=1}^{\infty} \frac1{k!} \Big( \frac{d^2 CC_2 \max\{|\tau|,|1-\tau|\}}{C_3} \Big)^{k}\Big)^2 < 1/2.
\end{align*}
This proves~\eqref{EqTheo326TesisDavid2}. Furthermore, by~\cite[Lemma 2.9(1)]{AJ} we have that for all $\ell \in \mathbb{N}$ there exists $C_{\ell}>0$ such that, for each $j$,
\[ |D^{\alpha}_x D^{\beta}_{\xi} q_j(x,\xi)| \leq C_{\ell} C_2^{|\alpha+\beta|} C_3^j \langle(x,\xi)\rangle^{-\rho(|\alpha+\beta|+2j)} e^{\ell \rho \varphi_{\omega}^{\ast}\big(\frac{|\alpha+\beta|+2j}{\ell}\big)} e^{|m|\omega(x,\xi)}, \]
for all $\alpha,\beta \in \mathbb{N}_0^d$ and $\langle(x,\xi)\rangle \geq R$ and, in particular, the estimate of Definition~\ref{DefFormalSums} follows.

Now, we extend $q_j(x,\xi)$ to $C^{\infty}(\mathbb{R}^{2d})$ for each $j \in \mathbb{N}_0$. To this aim, we take $\phi \in \mathcal{D}_{\sigma}(\mathbb{R}^{2d})$, supported in $\{ (x,\xi) \in \mathbb{R}^{2d} : \langle(x,\xi)\rangle \leq 2R \}$ and equal to $1$ when $\langle(x,\xi)\rangle \leq R$. Then, we set $\widetilde{q_j}(x,\xi) := q_j(x,\xi)(1-\phi)(x,\xi)$, which satisfies $\widetilde{q_j}=q_j$ if $\langle(x,\xi)\rangle > 2R$ and vanishes if $\langle(x,\xi)\rangle \leq R$. It is easy to see that $1-\phi \in \GS^{0,\omega}_{\rho}$. 
Hence, by Lemma~\ref{Lemma324TesisDavid}, $\widetilde{q_j}(x,\xi) \in \FGS^{|m|,\omega}_{\rho}$.

We identify $\widetilde{q_j}=q_{j}$ and we show that $\sum q_j \circ p \sim 1$. For $j>0$, by the definition of $q_j(x,\xi)$ we have
\begin{align*}
q_j(x,\xi)p(x,\xi) &= - \! \! \sum_{0<|\epsilon+\gamma| \leq j} \! \! \frac{(-1)^{|\epsilon|}}{\epsilon! \gamma!} \tau^{|\epsilon|} (1-\tau)^{|\gamma|} (\partial^{\gamma}_{\xi} D^{\epsilon}_x q_{j-|\epsilon+\gamma|}(x,\xi)) (\partial^{\epsilon}_{\xi} D^{\gamma}_x p(x,\xi)) \\
&= -r_j(x,\xi) + q_j(x,\xi)p(x,\xi),
\end{align*}
where $\sum r_j := \sum q_j \circ p$ (cf.~\cite[Proposition 4.13]{AJ}). Thus, $r_j(x,\xi)=0$ for $j>0$. Also, by the definition of composition, $r_0(x,\xi)=q_0(x,\xi)p(x,\xi)=1$ if $\langle(x,\xi)\rangle > 2R$, which shows that $\sum q_j \circ p \sim 1$. Since $\sum q_j$ is a formal sum in $\FGS^{|m|,\omega}_{\rho}$, by~\cite[Theorem 4.6]{AJ} there exists $q(x,\xi) \in \GS^{|m|,\omega}_{\rho}$ such that $q \sim \sum q_j$. Finally,~\cite[Proposition 4.14]{AJ} yields $q \circ p \sim 1$, and the proof is complete.
\end{proof}


\begin{cor}
\label{regularity}
Let $\omega$ be a weight function and let $\sigma$ be a weight function that satisfies ($\alpha_0$) with $\omega(t^{1/\rho}) = o(\sigma(t))$ as $t \to \infty$. If $p(x,\xi) \in \GS^{m,\omega}_{\rho}$ satisfies the hypotheses of Theorem~\ref{TheoParametrix}, any quantization of the corresponding pseudodifferential operator $P$ is $\omega$-regular.
\end{cor}
\begin{proof}	
By Theorem \ref{TheoParametrix} there is a pseudodifferential operator $Q$ such that $Q\circ P=I+R$, being $I$ the identity operator and $R$ an $\omega$-regularizing operator (as a direct consequence of Theorems~\ref{TheoCompositionTau} and \ref{TheoLarguisimoTau} for $\tau=0$). Then, $u=Q(Pu)-Ru\in \Sch_\omega(\R^d)$ for any $u\in \Sch'_\omega(\R^d)$ with $Pu\in\Sch_\omega(\R^d)$. The same argument is valid for an arbitrary quantization.
\end{proof}

\section{Global $\omega$-hypoellipticity for mixed classes}\label{SectionHypoellipticity}
In what follows, $m,m_0 \in \mathbb{R}$, $m_0\leq m$, $0<\rho\leq 1$, and for any given $\omega$ weight function, $\sigma$ is a Gevrey weight function, i.e. $\sigma(t)=t^a$, for some $0<a<1$, such that
\begin{equation}\label{EqDefWeightSigma}
\omega(t^{1/\rho}) = o(\sigma(t)), \qquad t \to \infty.
\end{equation}
\begin{defin}\label{DefHypoellipticity}
Let $a \in \GS^{m,\omega}_{\rho}$. We say that $a$ is an \emph{$\omega$-hypoelliptic symbol} in the class $\HGS^{m,m_0;\omega}_{\rho}$, and we write $a \in \HGS^{m,m_0;\omega}_{\rho}$, if there exist a Gevrey weight function $\sigma$ satisfying~\eqref{EqDefWeightSigma} and $R\geq 1$ such that
\begin{enumerate}
\item[(i)] There exist $C_1,C_2>0$ such that
\[ C_1 e^{m_0\omega(x,\xi)} \leq |a(x,\xi)| \leq C_2 e^{m\omega(x,\xi)}, \qquad \langle(x,\xi)\rangle \geq R. \]
\item[(ii)] There exist $C>0$, $n \in \mathbb{N}$ such that
\[ |D^{\alpha}_x D^{\beta}_{\xi} a(x,\xi)| \leq C^{|\alpha+\beta|} \langle(x,\xi)\rangle^{-\rho|\alpha+\beta|} e^{\frac1{n}\varphi^{\ast}_{\sigma}(n|\alpha|)} e^{\frac1{n}\varphi^{\ast}_{\sigma}(n|\beta|)} |a(x,\xi)|, \]
for $\langle(x,\xi)\rangle \geq R$, $ \ \alpha,\beta \in \mathbb{N}_0^d$.
\end{enumerate}
\end{defin}
We show in Theorem~\ref{TheoEllipticityEquivalent} below that Definition~\ref{DefHypoellipticity} is independent on the quantization $\tau$ for the case $m_0=m$. Hence, we extend~\cite[Proposition 8.4]{BBR1996Global}, showing that $\omega$-hypoelliptic symbol classes are not perturbed by a change of quantization. We observe that any pseudodifferential operator defined by an $\omega$-hypoelliptic symbol is also $\omega$-regular by Theorem~\ref{TheoParametrix}, but the converse is not true. For instance, the twisted Laplacian in $\mathbb{R}^2$,
\[ L=\Big(D_x-\frac{1}{2}y\Big)^2+\Big(D_y-\frac{1}{2}x\Big)^2 \]
is $\omega$-regular for every weight function $\omega$ as it is shown in~\cite[Example 5.4]{BJORegularity}, but its corresponding symbol is not $\omega$-hypoelliptic for any given weight function $\omega$ by~\cite[Remark 5.5]{BJORegularity}.

For technical reasons, the class of global symbols for which Theorem~\ref{TheoEllipticityEquivalent}  holds needs to be smaller than the one introduced in Section~\ref{SectionPreliminaries}. Namely, we need to introduce some kind of mixed conditions. The following is the corresponding definition for symbols:
\begin{defin}\label{mixed-symbol}
We say that $a \in \widetilde{\GS}^{m,\omega}_{\rho}$ if $a \in C^{\infty}(\mathbb{R}^{2d})$ and there exists a Gevrey weight function $\sigma$ satisfying \eqref{EqDefWeightSigma} such that for all $\lambda>0$ there is $C_{\lambda}>0$ with
\[ |D^{\alpha}_x D^{\beta}_{\xi} a(x,\xi)| \leq C_{\lambda} \langle(x,\xi)\rangle^{-\rho|\alpha+\beta|} e^{\lambda \varphi^{\ast}_{\sigma}\big(\frac{|\alpha+\beta|}{\lambda}\big)} e^{m\omega(x,\xi)}, \qquad \alpha,\beta \in \mathbb{N}_0^d, \ x,\xi \in \mathbb{R}^d. \]
\end{defin}
Definitions~\ref{DefHypoellipticity} and~\ref{mixed-symbol} are independent of the weight function $\sigma$, since given two Gevrey weight functions $\sigma_1$ and $\sigma_2$ with~\eqref{EqDefWeightSigma}, the Gevrey weight function $\sigma(t) := \min\{ \sigma_1(t), \sigma_2(t) \}$, $t>1$, satisfies~\eqref{EqDefWeightSigma} too.


According to condition~\eqref{EqDefWeightSigma}, we have, by~\cite[Lemma 2.9(1)]{AJ}, that for all $\lambda,\mu>0$ there exists 
$C>0$ such that
\begin{equation}\label{EqRelationOmegaSigma}
\lambda \varphi^{\ast}_{\sigma}\big(\frac{j}{\lambda}\big) \leq C + \mu \rho \varphi^{\ast}_{\omega}\big(\frac{j}{\mu}\big), \qquad j \in \mathbb{N}_0.
\end{equation}
As an immediate consequence we have $\widetilde{\GS}^{m,\omega}_{\rho} \subseteq \GS^{m,\omega}_{\rho}$.

\begin{lema}\label{LemmaUseful}
Let $a \in \widetilde{\GS}^{m,\omega}_{\rho}$. Then $a \in \HGS^{m,m;\omega}_{\rho}$ if and only if there exist $R \geq 1$ and $C'_1>0$ such that $|a(x,\xi)| \geq C'_1 e^{m\omega(x,\xi)}$ for $\langle(x,\xi)\rangle \geq R$.
\end{lema}
\begin{proof}
The necessity is obvious. For the sufficiency, since $a \in \widetilde{\GS}^{m,\omega}_{\rho}$, for $\sigma$ as in~\eqref{EqDefWeightSigma} there exists $C>0$ with
\begin{equation}\label{EqLemmaUseful1}
|D^{\alpha}_x D^{\beta}_{\xi} a(x,\xi)| \leq C \langle(x,\xi)\rangle^{-\rho|\alpha+\beta|} e^{\varphi^{\ast}_{\sigma}(|\alpha+\beta|)} e^{m\omega(x,\xi)}, \qquad \alpha,\beta \in \mathbb{N}_0^d, \ x,\xi \in \mathbb{R}^d,
\end{equation}
which in particular yields
\begin{equation}\label{aux}
C'_1 e^{m\omega(x,\xi)} \leq |a(x,\xi)| \leq C e^{m\omega(x,\xi)}, \qquad \langle(x,\xi)\rangle \geq R.
\end{equation}
This shows Definition~\ref{DefHypoellipticity}$(i)$. For condition $(ii)$, by~\eqref{EqEstimationVarphi}, $e^{\varphi^{\ast}_{\sigma}(|\alpha+\beta|)} \leq e^{\frac1{2}\varphi^{\ast}_{\sigma}(2|\alpha|)} e^{\frac1{2}\varphi^{\ast}_{\sigma}(2|\beta|)}$. Thus, by~\eqref{EqLemmaUseful1} and~\eqref{aux}, we have (since $C_1' \leq C$)
\[ |D^{\alpha}_x D^{\beta}_{\xi} a(x,\xi)| \leq \Big( \frac{C}{C'_1} \Big)^{|\alpha+\beta|} \langle(x,\xi)\rangle^{-\rho|\alpha+\beta|} e^{\frac1{2}\varphi^{\ast}_{\sigma}(2|\alpha|)} e^{\frac1{2}\varphi^{\ast}_{\sigma}(2|\beta|)} |a(x,\xi)|, \]
for $\langle(x,\xi)\rangle \geq R$, $\alpha,\beta \in \mathbb{N}_0^d$. Since $a \in \widetilde{\GS}^{m,\omega}_{\rho} \subseteq \GS^{m,\omega}_{\rho}$, the result follows.
\end{proof}

Similar mixed conditions are imposed to amplitudes and formal sums.
\begin{defin}\label{DefAmplitudeMix}
An amplitude $a(x,y,\xi) \in C^{\infty}(\mathbb{R}^{3d})$ belongs to $\widetilde{\GA}^{m,\omega}_{\rho}$ if there exists a Gevrey weight function $\sigma$ satisfying \eqref{EqDefWeightSigma} such that for all $\lambda>0$ there is $C_{\lambda}>0$ with
\[ |D^{\alpha}_x D^{\beta}_y D^{\gamma}_{\xi} a(x,y,\xi)| \leq C_{\lambda} \frac{\langle x-y \rangle^{\rho|\alpha+\beta+\gamma|}}{\langle(x,y,\xi)\rangle^{\rho|\alpha+\beta+\gamma|}} e^{\lambda \varphi^{\ast}_{\sigma}\big(\frac{|\alpha+\beta+\gamma|}{\lambda}\big)} e^{m\omega(x,\xi)}, \quad \alpha,\beta,\gamma \in \mathbb{N}_0^d, \ x,y,\xi \in \mathbb{R}^d. \]
\end{defin}

\begin{defin}
A formal sum $\sum p_j$ is in $\widetilde{\FGS}^{m,\omega}_{\rho}$ if $p_j \in C^{\infty}(\mathbb{R}^{2d})$ and there exist a Gevrey weight function $\sigma$ satisfying~\eqref{EqDefWeightSigma} and $R\geq 1$ such that for all $n \in \mathbb{N}$ there exists $C_n>0$ such that
\[ |D^{\alpha}_x D^{\beta}_{\xi} p_j(x,\xi)| \leq C_n \langle(x,\xi)\rangle^{-\rho(|\alpha+\beta|+j)} e^{n\varphi^{\ast}_{\sigma}\big(\frac{|\alpha+\beta|+j}{n}\big)} e^{m\omega(x,\xi)}, \]
for each $j \in \mathbb{N}_0$, $\alpha,\beta \in \mathbb{N}_0^d$, $\log\big(\frac{\langle(x,\xi)\rangle}{R}\big) \geq \frac{n}{j}\varphi^{\ast}_{\omega}\big(\frac{j}{n}\big)$.
\end{defin}

\begin{defin}\label{EqDefinEquivalenceFormalSums}
We say that $\sum a_j \sim \sum b_j$ in $\widetilde{\FGS}^{m,\omega}_{\rho}$ if there exist a Gevrey weight function $\sigma$ satisfying \eqref{EqDefWeightSigma} and $R\geq 1$ such that for all $n \in \mathbb{N}$ there exist $C_n>0$, $N_n \in \mathbb{N}$ such that
\[ \big|D^{\alpha}_x D^{\beta}_{\xi} \sum_{j<N} (a_j-b_j)\big| \leq C_n \langle(x,\xi)\rangle^{-\rho(|\alpha+\beta|+N)} e^{n\varphi^{\ast}_{\sigma}\big(\frac{|\alpha+\beta|+N}{n}\big)} e^{m\omega(x,\xi)}, \]
for all $N \geq N_n$, $\alpha,\beta \in \mathbb{N}_0^d$, $\log\big(\frac{\langle(x,\xi)\rangle}{R}\big) \geq \frac{n}{N}\varphi^{\ast}_{\omega}\big(\frac{N}{n}\big)$.
\end{defin}
Again by~\eqref{EqRelationOmegaSigma} it is also clear that $\widetilde{\GA}^{m,\omega}_{\rho} \subseteq \GA^{m,\omega}_{\rho}$ and $\widetilde{\FGS}^{m,\omega}_{\rho} \subseteq \FGS^{m,\omega}_{\rho}$.

The amplitudes introduced in Definition~\ref{DefAmplitudeMix} do not have exponential growth in the variable $y$ to avoid the increasing in the order $m \in \mathbb{R}$ in some results in Section~\ref{SectionSymbolicCalculus}. For instance, if $a \in \widetilde{\GA}^{m,\omega}_{\rho}$, then, following Example~\ref{ExampleFormalSumTau},
\begin{equation}\label{ExPjTauMix}
p_j(x,\xi) := \sum_{j=0}^{\infty} \sum_{|\beta+\gamma|=j} \frac1{\beta! \gamma!} \tau^{|\beta|} (1-\tau)^{|\gamma|} \partial^{\beta+\gamma}_{\xi} (-D_x)^{\beta} D^{\gamma}_y a(x,y,\xi) \Big|_{y=x} \in \widetilde{\FGS}^{m,\omega}_{\rho}.
\end{equation}
It is easy to check that $\varphi_j$ (defined in~\eqref{psi-function}) belongs to $\widetilde{\GS}^{0,\omega}_{\rho}$. Hence the corresponding symbolic calculus is developed in the same manner as for the global symbol class $\GS^{m,\omega}_{\rho}$. In particular, by~\cite[Theorem 4.6]{AJ}, we have, from~\eqref{ExPjTauMix},
\begin{equation}\label{EqTauSymbolMixing}
p_{\tau}(x,\xi) := \sum_{j=0}^{\infty} \varphi_j(x,\xi) p_j(x,\xi) \in \widetilde{\GS}^{m,\omega}_{\rho}
\end{equation}
for all $\tau \in \mathbb{R}$. Such symbol is called is the \emph{$\tau$-symbol of the pseudodifferential operator associated to the amplitude $a(x,y,\xi) \in \widetilde{\GA}^{m,\omega}_{\rho}$}. In addition, as a consequence of Theorem~\ref{TheoLarguisimoTau} we obtain Theorem~\ref{TheoRelationDifferentSymbols} for mixed classes.

\begin{theo}\label{TheoEquivalenceTauSymbolsTilde}
Let $\tau_1,\tau_2 \in \mathbb{R}$. If $a_{\tau_1}(x,\xi), a_{\tau_2}(x,\xi) \in \widetilde{\GS}^{m,\omega}_{\rho}$ are the $\tau_1$-symbol and the $\tau_2$-symbol of the pseudodifferential operator $A$, then
\[ a_{\tau_2}(x,\xi) \sim \sum_{j=0}^{\infty} \sum_{|\alpha|=j} \frac1{\alpha!} (\tau_1-\tau_2)^{|\alpha|} \partial^{\alpha}_{\xi} D^{\alpha}_x a_{\tau_1}(x,\xi) \]
in $\widetilde{\FGS}^{m,\omega}_{\rho}$.
\end{theo}

Now we are ready to prove the main theorem of this section.
\begin{theo}\label{TheoEllipticityEquivalent}
Let $\tau_1,\tau_2 \in \mathbb{R}$ and let $a_{\tau_1} \in \widetilde{\GS}^{m,\omega}_{\rho}$. If $a_{\tau_1} \in \HGS^{m,m;\omega}_{\rho}$, then $a_{\tau_2} \in \HGS^{m,m;\omega}_{\rho}$.
\end{theo}
\begin{proof}
By~\eqref{EqTauSymbolMixing} we have $a_{\tau_2} \in \widetilde{\GS}^{m,\omega}_{\rho}$. Therefore, by Lemma~\ref{LemmaUseful}, it is enough to show that there exist $R\geq 1$, $D>0$ such that
\begin{equation}\label{EqEstimateaTau2}
|a_{\tau_2}(x,\xi)| \geq D e^{m\omega(x,\xi)}
\end{equation}
for $\langle(x,\xi)\rangle \geq R$. In fact, by assumption, by the same result there are $R_1 \geq 1$, $D_1>0$ such that
\begin{equation}\label{EqEstimateaTau1}
|a_{\tau_1}(x,\xi)| \geq D_1 e^{m\omega(x,\xi)}
\end{equation}
for $\langle(x,\xi)\rangle \geq R_1$. By Theorem~\ref{TheoEquivalenceTauSymbolsTilde} and Definition~\ref{EqDefinEquivalenceFormalSums}, there exist a Gevrey weight function $\sigma_1$ satisfying \eqref{EqDefWeightSigma} and $R_2\geq 1$ such that 
there exist $C_1>0$, $N_1 \in \mathbb{N}$:
\[ \Big| a_{\tau_2}(x,\xi) - \sum_{j<N} \sum_{|\alpha|=j} \frac1{\alpha!} (\tau_1-\tau_2)^{|\alpha|} \partial^{\alpha}_{\xi} D^{\alpha}_x a_{\tau_1}(x,\xi) \Big| \leq C_1 \langle(x,\xi)\rangle^{-\rho N} e^{\varphi^{\ast}_{\sigma_1}(N)} e^{m\omega(x,\xi)} \]
for $N \geq N_1$ and $\log\big(\frac{\langle(x,\xi)\rangle}{R_2}\big) \geq \frac1{N}\varphi^{\ast}_{\omega}(N)$.
By~\eqref{EqRelationOmegaSigma}, there exists $A_1>0$ such that $\varphi^{\ast}_{\sigma_1}(N) \leq A_1 + \rho\varphi^{\ast}_{\omega}(N)$ for all $N\in\N$. 
Then,
\begin{equation}\label{EqTheoEllipticity2}
\Big| a_{\tau_2}(x,\xi) - \sum_{j<N} \sum_{|\alpha|=j} \frac1{\alpha!} (\tau_1-\tau_2)^{|\alpha|} \partial^{\alpha}_{\xi} D^{\alpha}_x a_{\tau_1}(x,\xi) \Big| \leq C_1 e^{A_1} R_3^{-\rho N} e^{m\omega(x,\xi)},
\end{equation}
for all $N\geq N_1$ and $\langle(x,\xi)\rangle \geq R_3 e^{\frac1{N} \varphi^{\ast}_{\omega}(N)}$, 
where $R_3 \geq R_2$ will be determined later.

We fix $N=N_1 \in \mathbb{N}$ and we claim that 
\begin{equation}\label{EqTheoHypoClaim1}
\Big| \sum_{j=0}^{N-1} \sum_{|\alpha|=j} \frac1{\alpha!} (\tau_1-\tau_2)^{|\alpha|} \partial^{\alpha}_{\xi} D^{\alpha}_x a_{\tau_1}(x,\xi) \Big| \geq \frac{D_1}{2} e^{m\omega(x,\xi)},
\end{equation}
if $\langle(x,\xi)\rangle$ is large enough. The inequality is immediate for $N=1$ by~\eqref{EqEstimateaTau1} for $\langle(x,\xi)\rangle \geq R_1$, so we shall assume that $N>1$. 
First, we estimate
\[ \Big| \sum_{j=1}^{N-1} \sum_{|\alpha|=j} \frac1{\alpha!} (\tau_1-\tau_2)^{|\alpha|} \partial^{\alpha}_{\xi} D^{\alpha}_x a_{\tau_1}(x,\xi) \Big|. \]
Since $a_{\tau_1}(x,\xi) \in \widetilde{\GS}^{m,\omega}_{\rho}$, there exists a Gevrey weight function $\sigma_2$ satisfying~\eqref{EqDefWeightSigma} such that 
there is $C_2>0$ with
\[ |D^{\alpha}_x D^{\alpha}_{\xi} a_{\tau_1}(x,\xi)| \leq C_2 \langle(x,\xi)\rangle^{-2\rho} e^{2\varphi^{\ast}_{\sigma_2}(N-1)} e^{m\omega(x,\xi)}, \]
for all $x, \xi \in \mathbb{R}^d$ and $1\leq |\alpha| \leq N-1$. Again by~\eqref{EqRelationOmegaSigma}, there exists $A_2>0$ such that $\varphi^{\ast}_{\sigma_2}(N-1) \leq A_2 + \rho\varphi^{\ast}_{\omega}(N-1)$. Consider $\langle(x,\xi)\rangle$ large enough so that
\[ \langle(x,\xi)\rangle \geq R_4 e^{\varphi^{\ast}_{\omega}(N-1)}, \]
with $R_4\geq 1$ to be determined. Then
\begin{align*}
|D^{\alpha}_x D^{\alpha}_{\xi} a_{\tau_1}(x,\xi)| &\leq C_2 e^{2A_2} \langle(x,\xi)\rangle^{-2\rho} e^{2\rho \varphi^{\ast}_{\omega}(N-1)} e^{m\omega(x,\xi)} \\
&\leq C_2 e^{2A_2} (R_4)^{-2\rho} e^{m\omega(x,\xi)},
\end{align*}
for $\langle(x,\xi)\rangle \geq R_4 e^{\varphi^{\ast}_{\omega}(N-1)}$, $1 \leq |\alpha| \leq N-1$. On the other hand, by formula~\cite[(0.3.1)]{NR2010global}, we obtain
\[ \sum_{j=1}^{N-1} \sum_{|\alpha|=j} \frac{|\tau_1-\tau_2|^{|\alpha|}}{\alpha!} \leq \sum_{j=1}^{N-1} \frac{(d|\tau_1-\tau_2|)^j}{j!}  \leq e^{d|\tau_1-\tau_2|}. \]
So, we deduce
\begin{equation}\label{EqTheoHypoClaim2}
\Big| \sum_{j=1}^{N-1} \sum_{|\alpha|=j} \frac1{\alpha!} (\tau_1-\tau_2)^{|\alpha|} \partial^{\alpha}_{\xi} D^{\alpha}_x a_{\tau_1}(x,\xi) \Big| \leq C_2 e^{2A_2} (R_4)^{-2\rho}  e^{d|\tau_1-\tau_2|} e^{m\omega(x,\xi)},
\end{equation}
for $\langle(x,\xi)\rangle \geq R_4 e^{\varphi^{\ast}_{\omega}(N-1)}$. Hence, by the triangular inequality, from formulas~\eqref{EqTheoHypoClaim2} and~\eqref{EqEstimateaTau1} we have
\begin{align*}
\Big| \sum_{j=0}^{N-1} \sum_{|\alpha|=j} \frac1{\alpha!} (\tau_1-\tau_2)^{|\alpha|} \partial^{\alpha}_{\xi} D^{\alpha}_x a_{\tau_1}(x,\xi) \Big| &\geq D_1 e^{m\omega(x,\xi)} - C_2 e^{2A_2} (R_4)^{-2\rho}  e^{d|\tau_1-\tau_2|} e^{m\omega(x,\xi)} \\
&\geq \frac{D_1}{2} e^{m\omega(x,\xi)},
\end{align*}
which shows~\eqref{EqTheoHypoClaim1} provided $R_4$ be so that
\[ (R_4)^{2\rho} \geq \frac{2}{D_1} C_2 e^{2A_2} e^{d|\tau_1-\tau_2|}, \]
and $\langle(x,\xi)\rangle \geq \max\{ R_1, R_4 e^{\varphi^{\ast}_{\omega}(N-1)}\}$. Finally we obtain, by~\eqref{EqTheoHypoClaim1} and~\eqref{EqTheoEllipticity2}, 
\[ |a_{\tau_2}(x,\xi)| \geq \frac{D_1}{2} e^{m\omega(x,\xi)} - C_1 e^{A_1} R_3^{-\rho N} e^{m\omega(x,\xi)} \geq \frac{D_1}{4} e^{m\omega(x,\xi)} \]
if $R_3^{\rho N} \geq \frac{4}{D_1} C_1 e^{A_1}$ and $\langle(x,\xi)\rangle \geq R:=\max\{ R_1, R_4 e^{\varphi^{\ast}_{\omega}(N-1)}, R_3 e^{\frac1{N} \varphi^{\ast}_{\omega}(N)} \}$. Then~\eqref{EqEstimateaTau2} is satisfied for $D=\frac{D_1}{4}>0$ and $R \geq 1$, and the proof is complete.
\end{proof}


\textbf{Acknowledgements.} The author was supported by the project GV Prometeo/2017/102. The author is indebted to David Jornet for his helpful comments and ideas, and for the careful reading of the paper. This is part of the author's Ph.D. Thesis.

\providecommand{\bysame}{\leavevmode\hbox to3em{\hrulefill}\thinspace}
\providecommand{\MR}{\relax\ifhmode\unskip\space\fi MR }
\providecommand{\MRhref}[2]{%
  \href{http://www.ams.org/mathscinet-getitem?mr=#1}{#2}
}
\providecommand{\href}[2]{#2}

\noindent
{Instituto Universitario de Matem\'atica Pura y Aplicada IUMPA\\
Universitat Po\-li\-t\`ecni\-ca de Val\`encia\\
Camino de Vera, s/n\\
E-46071 Valencia\\
Spain}

\vspace{0.2cm}
\noindent
\Letter $ \ $ \href{viaslo@upv.es}{viaslo@upv.es}

\end{document}